\documentclass[preprint,3p,sort]{elsarticle}

\usepackage{amsfonts}
\usepackage{amsthm}
\usepackage{graphicx}
\usepackage{epstopdf}
\usepackage{algorithmic}
\usepackage{amsmath}
\usepackage{amssymb}
\usepackage{hyperref}
\usepackage[T1]{fontenc}
\usepackage{soul,color}
\usepackage{cases}
\newtheorem{theorem}{Theorem}[section]
\newtheorem{proposition}[theorem]{Proposition}
\newtheorem{lemma}[theorem]{Lemma}

\newtheorem{definition}{Definition}[section]
\theoremstyle{remark}
\newtheorem{remark}[theorem]{Remark}
\theoremstyle{example}
\newtheorem{example}[theorem]{Example}

\numberwithin{equation}{section}
\begin{document}


\begin{frontmatter}

%

\title{Nonlinear Boundary Conditions for Initial Boundary Value Problems with Applications in Computational Fluid Dynamics}

\author[sweden,southafrica]{Jan Nordstr\"{o}m}
\cortext[secondcorrespondingauthor]{Corresponding author}
\ead{jan.nordstrom@liu.se}
\address[sweden]{Department of Mathematics, Applied Mathematics, Link\"{o}ping University, SE-581 83 Link\"{o}ping, Sweden}
\address[southafrica]{Department of Mathematics and Applied Mathematics, University of Johannesburg, P.O. Box 524, Auckland Park 2006, Johannesburg, South Africa}

\begin{abstract}
We derive new boundary conditions and implementation procedures for nonlinear initial boundary value problems (IBVPs) with non-zero boundary data that lead to bounded solutions. The new boundary procedure is applied to nonlinear IBVPs in skew-symmetric form, including dissipative terms.
The complete procedure has two main ingredients. In the first part (published in \cite{nordstrom2022linear-nonlinear,Nordstrom2022_Skew_Euler}), the energy and entropy rate in terms of a surface integral with boundary terms was produced for problems with first derivatives. 
In this second part we complement it by adding second derivative terms and new nonlinear boundary procedures leading for boundary conditions with non-zero data. \textcolor{red}{The new nonlinear boundary procedure generalise the well known characteristic boundary procedure for linear problems to the nonlinear setting.}

\textcolor{green}{To introduce the procedure, a skew-symmetric scalar IBVP encompassing the linear advection equation and Burgers equation is analysed. Once the continuous analysis is done, we show that energy stable nonlinear discrete approximations follow by using summation-by-parts operators combined with weak boundary conditions. The scalar analysis is subsequently repeated for general nonlinear systems of equations.
Finally, the new boundary procedure is applied to four important IBVPs in computational fluid dynamics: the incompressible Euler and Navier-Stokes, the shallow water and the compressible Euler equations.}
\end{abstract}

\begin{keyword}
Nonlinear boundary conditions \sep Navier-Stokes equations \sep Euler equations \sep shallow water equations  \sep  energy and entropy stability  \sep summation-by-parts
\end{keyword}


\end{frontmatter}


\section{Introduction}
\textcolor{red}{A proper treatment of boundary conditions is a crucial and often the most difficult aspect for bounding the solution in initial boundary value problems (IBVPs). For nonlinear IBVPs this aspect becomes even more crucial since nonlinearities have a tendency to enhance growth of instabilities. Generally speaking one finds two historical avenues  to obtain estimates of solutions IBVPs. The first one employs the energy method \cite{kreiss1970,kreiss1989initial,Gustafsson1978,gustafsson1995time,oliger1978,nordstrom2005,nordstrom2019} while the second one uses the entropy stability theory \cite{tadmor1984,Tadmor1987,Tadmor2003,godunov1961interesting,volpert1967,kruzkov1970,dafermos1973entropy,lax1973,harten1983}. Traditionally, the energy method has been applied to linearised versions of systems of hyperbolic equations in order to develop boundary treatments that lead to energy estimates. In practice, these boundary conditions are needed to develop energy stable numerical approximations that weakly impose boundary information, e.g., through penalty terms \cite{carpenter1994time,nordstrom_roadmap} or numerical flux functions \cite{hindenlang2019,leveque1998,ESDGSEM2D_paper,xing2014}. The entropy method has often been applied to nonlinear hyperbolic systems on domains without boundary conditions (as for periodic boundary conditions or infinite domains). However, entropy stability theory has also been used for the development of boundary conditions \cite{dalcin2019conservative,dubois1988,hindenlang2019,parsani2015entropy,svard2012,svard2014,svard2021entropy,svar2018,svard2021,chan2022,Gjesteland2022}, but almost exclusively for solid walls (or glancing boundaries).
Solid wall boundary conditions are notoriously simple and straightforward to apply and implement due to their homogeneous nature, i.e. no external non-zero data must be considered. In contrast to previous investigations, we will for the first time (to the best of our knowledge) treat the general case with non-homogeneous nonlinear boundary conditions and derive estimates of the solution in terms of given data.}

This paper constitutes the second part of the stability theory for nonlinear IBVPs on skew-symmetric form initiated in \cite{nordstrom2022linear-nonlinear,Nordstrom2022_Skew_Euler}. In the first part, energy rates in terms of surface integrals with boundary terms were derived using skew-symmetric formulations.  Here in the second part we add on second derivative terms and develop a nonlinear boundary procedure which controls the boundary terms. In fact, we generalise the well known linear characteristic boundary procedure  \cite{kreiss1970,MR436612} and prove that we can bound both strong and weak boundary conditions in terms of given external data. We stress that this second part (the new boundary procedure) in itself does not depend on the first part. It can be used by other preferred formulations of nonlinear IBVPs, as long as similar settings are obtained. Most of the analysis in this second part is limited to IBVPs involving first derivatives (essentially hyperbolic problems). However, the new nonlinear boundary procedure can also be used to bound dissipative IBVPs involving second derivatives  similar to what has been done in the linear case \cite{nordstrom2005,MR1339182,MR1669660}. We show how to formally extend the analysis for the first derivative case to the one including second derivatives (essentially parabolic and incompletely parabolic problems).

We exemplify the procedure on four important IBVPs in computational fluid dynamics (CFD):  the incompressible Euler equations (IEEs) and skew-symmetric formulations of the shallow water equations (SWEs) and the compressible Euler equations (CEEs) and the incompressible Navier-Stokes equations (INSEs).
It was shown in \cite{nordstrom2022linear-nonlinear} that the original form of the velocity-divergence form of the IEEs equations had the required skew-symmetric form and that it could be derived for the SWEs. In  \cite{Nordstrom2022_Skew_Euler} we showed that also the CEEs could be transformed to skew-symmetric form with a quadratic mathematical (or generalised) entropy by using specific variables involving square roots of density and pressure, also used in \cite{Petr2007,Rozema2014,Sesterhenn2014,Halpern2018}. 
The continuous procedure in \cite{nordstrom2022linear-nonlinear,Nordstrom2022_Skew_Euler}), was reused by discretising the equations in space using summation-by-parts (SBP) operators \cite{svard2014review,fernandez2014review} which discretely mimic the integration-by-parts (IBP) procedure used in the continuous analysis. A provably stable nonlinear approximation was obtained by a priori {\it assuming}  that the boundary treatment was dissipative. In this paper we close that particular knowledge gap and develop provably stable boundary procedures for nonlinear systems of equation with non-zero boundary data.

\textcolor{green}{Using lifting approaches \cite{Arnold20011749,nordstrom_roadmap} and proper continuous boundary conditions for the IBVP, it is straightforward to apply results (formulations that lead to energy conservation and bounds) from the continuous analysis and develop stable numerical schemes. The only requirement is that one can formulate the numerical procedure on summation-by-parts (SBP) form with weak boundary conditions on simultaneous approximation term (SAT) form \cite{svard2014review,fernandez2014review} or equivalently through numerical flux functions \cite{kopriva2021}. This SBP-SAT procedure enables research groups using different numerical techniques  such as finite difference \cite{nordstrom2009stable,svard2007stable,svard2008stable}, finite volume \cite{nordstrom2012weak,nordstrom2003finite}, spectral elements \cite{carpenter2014entropy,carpenter1996spectral}, flux reconstruction \cite{castonguay2013energy,huynh2007flux,huynh2007flux}, discontinuous Galerkin \cite{gassner2013skew,hesthaven1996stable,kopriva2021} and continuous Galerkin schemes \cite{abgrall2020analysis,abgrall2021analysis} to develop stable nonlinear schemes in a systematic way.}

\textcolor{blue}{We close the paper by discussing some open questions for boundary conditions to systems of nonlinear IBVPs. In this paper we derive boundary conditions that lead to energy estimates and subsequently to energy stability using the SBP-SAT technique. For linear systems of IBVPs, the character of the boundary term is given by external data and energy boundedness leads more or less directly to both uniqueness and existence if a minimal number of boundary conditions are employed in the bound \cite {nordstrom2020}. In the nonlinear case this is more complicated since multiple solution dependent forms of the boundary term may exist  \cite{nordstrom2022linear-nonlinear,Nordstrom2022_Skew_Euler}. This has implications for the number of boundary conditions to apply, and also for uniqueness and existence.}

The paper is organised as follows: \textcolor{green}{We start by going through the fundamental parts of the paper for a scalar one-dimensional IBVP in Section \ref{sec:illustration_theory} which provide an extended introduction and a roadmap for the remaining part of the paper in a simplified setting. The more easily accessible scalar analysis is related to the analysis of the more demanding nonlinear system analysis in Sections 3-5.} In Section~\ref{sec:theory} we reiterate and complement the main content in the previously published first part of procedure \cite{nordstrom2022linear-nonlinear,Nordstrom2022_Skew_Euler}. The new second part: how to formulate and impose general boundary conditions with non-zero data for nonlinear systems of equations such that a bound is obtained, is presented in Section~\ref{BC_theory}.  In Section~\ref{numerics} we show that the energy bounded continuous formulation lead to  nonlinear stability of an SBP-SAT based scheme, including non-zero boundary data. Section~\ref{Examples} include examples of the new general nonlinear boundary procedures  applied to important IBVPs in CFD. Explicit examples of boundary conditions and implementation procedures are given for the IEEs, the SWEs, the CEEs and the INSEs. \textcolor{blue}{We discussing some open questions related to boundary conditions for nonlinear IBVPs in Section \ref{sec:open_q}.} A summary is provided in Section~\ref{sec:conclusion}.

\textcolor{red}{\section{Part 0: Illustration of the main concepts in a scalar setting: a roadmap for the paper}\label{sec:illustration_theory}
As a prolonged introduction of this paper, we will present the fundamental parts of the development for a scalar one-dimensional IBVP. The scalar analysis avoid most of the technical difficulties with nonlinear systems of equation in multiple dimensions, but retain and clarify the fundamental ones. In particular we seek to highlight the differences between linear and nonlinear boundary treatment and  present  the energy analysis.
 We focus on the governing IBVP and its energy boundedness and thereafter implemented the finding in an SBP-SAT scheme that lead to nonlinear energy stability. The scalar analysis in this section is subsequently repeated for fully nonlinear systems of equations with significantly more technical difficulties in the upcoming sections. The three subsections in this scalar section has the same titles as the three following sections for the nonlinear systems of equations in order to simplify comparisons.
\subsection{Part 1: Producing energy and entropy rates}\label{sec:theory_il}
We considering the following IBVP with the governing on skew-symmetric form
\begin{equation}\label{eq:nonlin_il}
u_t + (a u)_x+au_x=0,  \quad t \geq 0,  \quad  x  \geq 0, 
\end{equation}
augmented with the initial condition $u(x,0)=f(x)$ and  the non-homogeneous general boundary condition
\begin{equation}\label{eq:nonlin_BC_il}
b(u) = g(t),  \quad t \geq 0,  \quad  x=0.
\end{equation}
In (\ref{eq:nonlin_il}),  $a$ is a general function which is solution dependent if we consider a nonlinear problem. Our prime examples are the nonlinear Burgers equation where $a=u(x,t)/3$ is solution dependent and the advection equation where  $a=c(x,t)/2$ and $c$ is a given function of $x,t$. The boundary operator $b$ in (\ref{eq:nonlin_BC_il}) can in similar manner as the governing equation be linear or nonlinear. The initial data $f$,  the boundary data $g$ and (in the linear case) the wave speed  $c(x,t)$ are given external smooth bounded functions which constitute input to the IBVP. 
Furthermore, we also consider $u$ to be smooth. The smoothness in the nonlinear setting could stem from a neglected dissipative term or from considering the IBVP for short times.  We only consider the boundary $x=0$ and hence the initial data $f$ has compact support and we assume that $u \rightarrow 0$ as $x \rightarrow  \infty$. 
\begin{remark}\label{explain_1} 
It is important to note that  $a=u/3$ is not a priori bounded while $a=c/2$ is.  
\end{remark}
We start by applying the energy method (multiplying the equation with $u$ and integrate-by-parts) to the skew-symmetric equation (\ref{eq:nonlin_il}). 
By integrating the second term and leaving the last one untouched we find 
\begin{equation}\label{eq:boundaryPart1_intermid_il}
\int_0^{\infty} u u_t dx + (a u^2)_{x=0}^{x \rightarrow  \infty} - \int_0^{\infty} u_x (au) dx + \int_0^{\infty} u (a u_x) dx =0.
\end{equation}
By ignoring the right boundary term, using the notation  $\|u\|^2= \int_0^{\infty} u^2 dx$ and observing that the volume terms above vanish (due to the skew-symmetry) we obtain the energy rate 
\begin{equation}\label{eq:boundaryPart1_il}
\frac{1}{2} \frac{d}{dt}\|u\|^2 - (a u^2)_{x=0}=0.  
\end{equation}
Our {\it first observation} is that the skew-symmetric form of equation (\ref{eq:nonlin_il}) leaves the energy rate to be determined by the boundary terms only. (The skew-symmetry for nonlinear systems of IBVPs were discussed in \cite{nordstrom2022linear-nonlinear,Nordstrom2022_Skew_Euler} and will be shortly summarised in Section \ref{sec:theory}). Next, we take $a>0$ in the whole computational domain such that $x=0$ is an inflow boundary that require a boundary condition. 
The  {\it second observation} is that specifying only $u$ in (\ref{eq:boundaryPart1_il}) does not necessarily bound the energy. In addition $a$ must be bounded which it is for linear problems but not in general for nonlinear ones. The  {\it second observation} implies that the complete boundary term must be considered when constructing the boundary operator $b(u)$ in the nonlinear case. 
\begin{remark}\label{explain_cubic} 
Our two examples illustrate the significance of the {\it second observation}. For the advection equation, $a=c(x,t)/2$ leads to a boundary term quadratic in $u$ multiplied with the bounded externally given coefficient $a$ with a specific sign. Consequently, we have a priori all information upon which to base the boundary treatment. For the Burgers equation where $a=u/3$ we end up with a boundary term cubic in $u$, where the coefficient $a$ is neither bounded nor have a definite sign upon which to base the boundary treatment. The nonlinear boundary procedure must control both the size and the sign of the coefficient.
\end{remark} 
\begin{remark}\label{explain_2} 
In the upcoming analysis for systems of equations, the boundary term is a vector-matrix-vector product of the form $U^T A(U) U$. Controlling the coefficient $a$ in the nonlinear scalar case discussed in Remark \ref{explain_cubic} correspond to controlling the eigenvalues and eigenvectors of the matrix $A(U)$ in the system case.
\end{remark}
Next, we show that the skew-symmetric form of (\ref{eq:nonlin_il}) also supports an entropy conservation law. By multiplying (\ref{eq:nonlin_il}) with the (entropy) variable $u$ we obtain
\begin{equation}\label{eq:entropy-derail}
 u u_t  +  u(au)_x + (au)u_x =0,
\end{equation}
which by introducing the  entropy $\phi=u^2/2$ and entropy flux  $\psi= au^2$ leads to the entropy conservation law 
\begin{equation}\label{eq:entropy-equation_il}
\phi_t + \psi_x = 0.
\end{equation}
\begin{remark}\label{explain_entropy_burgers}
For Burgers equation, the quadratic entropy $\phi=u^2/2$ and the cubic entropy flux $\psi= u^3/3$  (the so called entropy-entropy flux pair) are well known \cite{Tadmor2003} and illustrates the concept of a mathematical (or generalised) entropy. The mathematical entropy contracts the governing equation (\ref{eq:nonlin_il})  to the conservation form (\ref{eq:entropy-equation_il}). It is not necessarily connected to a thermodynamic entropy. Equations (\ref{eq:nonlin_il}) and (\ref{eq:entropy-equation_il}) reveals that also the total entropy in the domain is bounded with proper boundary conditions. 
\end{remark}
\subsection{Part 2: The nonlinear boundary conditions and weak implementation procedure}\label{BC_theory_il} 
The final target for this analysis is a provably energy stable numerical implementation which in general require summation-by-parts (SBP) operators and weak boundary conditions implemented using penalty terms (the so called SAT formalism), or equivalently, numerical fluxes. Following \cite{nordstrom_roadmap}, we derive the conditions for energy boundedness in the continuous setting and mimic that numerically by constructing an energy stable SBP-SAT scheme. Equation (\ref{eq:nonlin_il}) augmented with a lifting operator for the boundary condition (\ref{eq:nonlin_BC_il}) is 
\begin{equation}\label{eq:nonlin_lif_il}
u_t + (a u)_x+au_x + l_c(2\sigma(b(u)-g))=0,  \quad t \geq 0,  \quad  x  \geq 0,
\end{equation}
where $\sigma$ is a yet undetermined penalty parameter. The lifting operator $l_c$ for two smooth functions  $\phi, \psi$ satisfies $\int\limits \phi   l_c(\psi) dx = (\phi  \psi)_{x=0}$. The lifting operator mimics the SAT term and enables development of the essential parts of the numerical boundary procedure already in the continuous setting \cite{Arnold20011749,nordstrom_roadmap}. 
\begin{remark}\label{explain_3} 
The use of lifting operators in the continuous scalar analysis, where the technical difficulties are limited, is not absolutely necessary. However, we use it here since it connects to the discrete analysis with SAT terms and also explain the upcoming procedure for the nonlinear systems of equations.
\end{remark}
By once more applying the energy method, now to (\ref{eq:nonlin_lif_il}),  we find that the energy rate becomes 
\begin{equation}\label{eq:boundaryPart1_il_pen}
\frac{1}{2} \frac{d}{dt}\|u\|^2 - a u^2+2\sigma u(b(u)-g)=0,  
\end{equation}
where again all terms are evaluated at $x=0$. To get an energy rate that lead to an estimate in terms of data,  we need to choose an appropriate boundary operator $b$ a the related penalty parameter $\sigma$.}

\textcolor{red}{Let us first choose $b=u$, i.e. the boundary condition $u=g_1$ and $\sigma = a$. This leads to 
\begin{equation}\label{eq:boundaryPart1_il_pen_lin}
\frac{1}{2} \frac{d}{dt}\|u\|^2 = a u^2-2 a u(u-g_1)=ag_1^2 - a(u-g_1)^2, 
\end{equation}
where we also added and subtracted $a g_1^2$. This is the classical linear result, which gives a bound in terms of boundary data $g_1$, wave speed $a=c/2$ and an additional dissipative term. However, this bound only holds if  $a$ is bounded, as in the advection case with an externally given wave speed. In the nonlinear case, for examplein the Burgers equation where $a=u/3$, the bound breaks down. The unbounded energy rate in (\ref{eq:boundaryPart1_il_pen_lin}) indicates that the wave speed must be included in the constructing of the boundary operator $b$.}

\textcolor{red}{Let us next choose $b=\sqrt{a} u$, i.e. the boundary condition $\sqrt{a} u=g_2$  and $\sigma = \sqrt{a} $. This leads to 
\begin{equation}\label{eq:boundaryPart1_il_pen_nonlin}
\frac{1}{2} \frac{d}{dt}\|u\|^2 = a u^2-2 \sqrt{a}  u(\sqrt{a} u-g_2)=g_2^2 - (\sqrt{a}u-g_2)^2,
\end{equation}
which is a valid estimate for  both the linear and nonlinear case. The energy rate  is now bounded by data and a dissipative term. In the new boundary procedure, we have controlled the possible unbounded growth.
\begin{remark}\label{explain_4} 
In the upcoming analysis for nonlinear systems of equations with boundary terms of the form $U^T A(U) U$, we will need to scale the characteristic variables with the square root of the eigenvalues in a similar manner. As in the scalar case we will also need a penalty matrix that relates to these square roots.
\end{remark}
\subsection{Part 3: Semi-discrete nonlinear stability}\label{numerics_il} 
As the last part of the scalar analysis (which form the last part of the introduction of this article) we show how the derived boundary conditions and penalty parameters that lead to a nonlinear energy bound also lead to nonlinear energy stability. First we introduce one-dimensional SBP difference operators as  $D=P^{-1}Q$ where $P$ is a positive definite diagonal quadrature matrix and  $Q$ an almost skew-symmetric matrix that satisfies the SBP constraint $Q+Q^T=B=diag[-1,0,...,0,1]$. The difference operator $D$  operating on the vector $V$ with a smooth function $v$ injected in the node points produces an approximate derivative as $DV \approx V_x $. All matrices have appropriate sizes such that all upcoming matrix-vector operations are defined.}

\textcolor{red}{To formulate the scheme we also need the diagonal matrix $A$ with the function values of $a$ injected on the diagonal, the solution vector $U=(u_0,u_1,...,u_N)$, The boundary operator $B(U)=(b(u),0,...,0)$, the data vector $G=(g,0,...,0)$ and the projection matrix $E=diag(1_,0,...,0)$. With this notation in place, the semi-discrete skew-symmetric scheme mimicking (\ref{eq:nonlin_lif_il}) becomes
\begin{equation}\label{eq:nonlin_lif_il_disc}
U_t + D(A U)+ADU + P^{-1} 2 \sigma E (B(U)-G)=0,
\end{equation}
where the last term on the lefthand side is the discrete lifting operator implementing the boundary condition weakly with a SAT term.  We apply the discrete energy method by left-multiplying (\ref{eq:nonlin_lif_il_disc}) with $U^T P$ to get
\begin{equation}\label{eq:nonlin_lif_il_disc_1}
\frac{1}{2} \frac{d}{dt}\|U\|^2_P + U^T (QA)U+U^T (AQ)U + 2 \sigma u_0 (b(u)_0 -g)=0,
\end{equation}
where we have used the notation $\|U\|^2_P=U^T P$  and the fact that $PA=AP$ since they are both diagonal. Next we integrate (\ref{eq:nonlin_lif_il_disc_1}) numerically by using the SBP operation $Q=-Q^T+B$ on the second term to get
\begin{equation}\label{eq:nonlin_lif_il_disc_2}
\frac{1}{2} \frac{d}{dt}\|U\|^2_P -(DU)^T P (AU)+(AU)^T P (DU) - (au^2) + 2 \sigma u(b(u) -g) =0,
\end{equation}
where we ignored the right boundary term as in the continuous setting. The volume terms on the lefthand side cancel exactly as in the continuous setting (\ref{eq:boundaryPart1_intermid_il}) and we are left with the same lefthand side as in (\ref{eq:boundaryPart1_il_pen}). Consequently the boundary operator and penalty parameter in (\ref{eq:boundaryPart1_il_pen}) also lead to nonlinear energy stability. 
\begin{remark}\label{explain_5} 
The above procedure illustrate the strength of the roadmap procedure in \cite{nordstrom_roadmap}. Once the energy estimates are obtained for the continuous IBVP, discrete stability follows automatically by using the SBP-SAT formulation. The same boundary conditions, penalty parameters, lifting operator, etc can be used. The derivation can be performed on the IBVP side and thereafter mimicked on the discrete side.
\end{remark}} 

\textcolor{green}{We end this section by showing that nonlinear stability guarantee that the eigenvalues of the nonlinear spatial operator reside on the right side in the complex plane, prohibiting exponential growth. We rewrite (\ref{eq:nonlin_lif_il_disc}) using the stable boundary operator $b=\sqrt{a} u$ and penalty coefficient $\sigma = \sqrt{a} $ in matrix-vector form as
\begin{equation}\label{eq:nonlin_lif_il_disc_spec1}
U_t + D(A U)+ADU + P^{-1} 2  \sqrt{A}  E ( \sqrt{A} U-G)=0.
\end{equation}
Next we cancel $G$ and observe that  $P^{-1},\sqrt{A},A$ and $E$ all commute since they are diagonal. This leads to 
\begin{equation}\label{eq:nonlin_lif_il_disc_spec2}
U_t + P^{-1}((Q+E)A + A(Q+E))U=0.
\end{equation}
We can now prove that the spatial operator $P^{-1}((Q+E)A + A(Q+E))$ has only eigenvalues with positive real parts.
The eigenvalues $\lambda$ and eigenvector $x$ are given by
\begin{equation}\label{eq:nonlin_lif_il_disc_spec3}
P^{-1}((Q+E)A + A(Q+E))x=\lambda x \quad  \Rightarrow \quad x^*((Q+E)A + A(Q+E))x=\lambda x^*Px = \lambda \|x\|^2_P,
\end{equation}
where $x^*$ is the complex conjugated transpose of $x$. The right relation above added to its transpose yield 
\begin{equation}\label{eq:nonlin_lif_il_disc_spec4}
x^*((Q+Q^T+2E)A + A(Q+Q^T+2E))x= 2 Re(\lambda) \|x\|^2_P
\end{equation}
which proves that the real part of the eigenvalues, $Re(\lambda)$, are non-negative since $Q+Q^T+2E=diag(1,0....0,1)$.}


\section{Part 1: Producing the energy and entropy rates}\label{sec:theory}
In this section we  repeat the findings in \cite{nordstrom2022linear-nonlinear,Nordstrom2022_Skew_Euler} to make the presentation self contained and to set the stage for the new upcoming boundary condition analysis.  \textcolor{red}{The system analysis in this section correspond to the scalar analysis in Section \ref{sec:theory_il}}. We also complement the governing equations with viscous fluxes to include e.g. the Navier-Stokes equations.
Following \cite{nordstrom2022linear-nonlinear,Nordstrom2022_Skew_Euler}, we consider the following IBVP
\begin{equation}\label{eq:nonlin}
P U_t + (A_i(V) U)_{x_i}+A_i^T(V)U_{x_i}+C(V)U= \epsilon (D_i(U))_{x_i},  \quad t \geq 0,  \quad  \vec x=(x_1,x_2,..,x_k) \in \Omega
\end{equation}
augmented with the initial condition $U(\vec x,0)=F(\vec x)$ in $\Omega$ and  the non-homogeneous boundary condition
\begin{equation}\label{eq:nonlin_BC}
B(V) U = G(\vec x,t),  \quad t \geq 0,  \quad  \vec x=(x_1,x_2,..,x_k) \in  \partial\Omega.
\end{equation}
In (\ref{eq:nonlin_BC}), $B$ is the boundary operator and $G$ the boundary data. In (\ref{eq:nonlin}), Einsteins summation convention is used and $P$ is a symmetric positive definite (or semi-definite) time-independent matrix that defines an energy norm (or semi-norm) $\|U\|^2_P= \int_{\Omega} U^T P U d\Omega$. Note in particular that $P$  is not a function of the state $U$. The $n \times n$ matrices $A_i,C$ are smooth functions of the $n$ component vector $V$. The viscous fluxes containing first derivatives are included in  $D_i$ and $\epsilon$ (typically the inverse of the Reynolds number) is a parameter measuring the influence of viscous forces. Furthermore, we assume that $U$ and $V$ are smooth. The smoothness could be considered either as a result of the dissipative terms in (\ref{eq:nonlin}) for $\epsilon \neq 0$ or stemming from (\ref{eq:nonlin}) for short times with smooth initial data when $\epsilon \rightarrow 0$.  
 \textcolor{red}{Note that (\ref{eq:nonlin}) and (\ref{eq:nonlin_BC}) encapsulates both linear ($V \neq U$) and nonlinear  ($V=U$) problems. This corresponds to the linear case with given wave speed for the advection equation and the solution dependent wave speed in the Burgers equation as discussed in Section \ref{sec:theory_il}.}

\subsection{Energy and entropy rates when $\epsilon \rightarrow 0$}\label{sec:theory_inviscid}
We begin with a formal definition and proposition for $\epsilon \rightarrow 0$ that was already published in  \cite{nordstrom2022linear-nonlinear,Nordstrom2022_Skew_Euler}.
\begin{definition}
Firstly, the problem (\ref{eq:nonlin}) with $\epsilon = 0$ is energy conserving if $\|U\|^2_P= \int_{\Omega} U^T P U d\Omega$ only changes due to boundary effects. Secondly, it is energy bounded if $\|U\|^2_P  \leq \|F\|^2_P$ for a minimal number of homogeneous $(G=0)$ boundary conditions (\ref{eq:nonlin_BC}). 
Thirdly, it is strongly energy bounded if $\|U\|^2_P  \leq \|F\|^2_P+  \int_0^t (\oint G^TG \ ds) dt$ for a minimal number of non-homogeneous $(G \neq 0)$ boundary conditions (\ref{eq:nonlin_BC}).
\end{definition}
 \textcolor{red}{\begin{remark}\label{explain_min_nr} 
A minimal number of dissipative boundary conditions in the linear case  leads to uniqueness by the fact that it determines the normal modes of the solution \cite{kreiss1970,Strikwerda1977797}. The (same) minimal number of boundary conditions can also be obtained using the energy method and specifying the number of boundary conditions required for a bound, see \cite{nordstrom2020}. 
For linear IBVPs,  the number of boundary conditions is independent of the solution and only depend on known external data. For nonlinear IBVPs, we here {\it assume} that the same situation holds, but the situation is not completely clear and we discuss this further in Section \ref{sec:open_q_uniqueness} below. 
\end{remark}} 
\begin{proposition}\label{lemma:Matrixrelation}
The IBVP  (\ref{eq:nonlin}) with $\epsilon = 0$ for is energy conserving if
\begin{equation}\label{eq:boundcond}
C+C^T = 0.
\end{equation}
It is energy bounded if it is energy conserving and the boundary conditions (\ref{eq:nonlin_BC}) for $G=0$ lead to
\begin{equation}\label{1Dprimalstab}
\oint\limits_{\partial\Omega}U^T  (n_i A_i)   \\\ U \\\ ds = \oint\limits_{\partial\Omega} \frac{1}{2} U^T ((n_i A_i)  +(n_i A_i )^T) U \\\ ds \geq 0.
\end{equation}
It is strongly energy bounded if it is energy conserving and the boundary conditions (\ref{eq:nonlin_BC}) for $G \neq 0$ lead to
\begin{equation}\label{1Dprimalstab_strong}
\oint\limits_{\partial\Omega}U^T  (n_i A_i)   \\\ U \\\ ds = \oint\limits_{\partial\Omega} \frac{1}{2} U^T ((n_i A_i)  +(n_i A_i )^T) U \\\ ds \geq - \oint\limits_{\partial\Omega} G^TG \\\ ds,
\end{equation}
where $G=G(\vec x,t)$ is independent of the solution $U$.
\end{proposition}
\begin{proof}
The energy method  \textcolor{red}{(multiply with $U^T$ and integrate over domain)}  applied to (\ref{eq:nonlin}) with $\epsilon = 0$ yields
\begin{equation}\label{eq:boundaryPart1}
\frac{1}{2} \frac{d}{dt}\|U\|^2_P + \oint\limits_{\partial\Omega}U^T  (n_i A_i)  \\\ U \\\ ds= \int\limits_{\Omega}(U_{x_i}^T  A_i U - U^T A_i^T U_{x_i}) \\\ d \Omega -\int\limits_{\Omega} U^T  C U \\\ d \Omega,
\end{equation}
where $(n_1,..,n_k)^T$ is the outward pointing unit normal. The terms on the right-hand side of (\ref{eq:boundaryPart1}) cancel by the skew-symmetry of the equation (as in the scalar case (\ref{eq:boundaryPart1_intermid_il})) and (\ref{eq:boundcond}) leading to energy conservation. If also (\ref{1Dprimalstab}) or (\ref{1Dprimalstab_strong}) holds, an energy bound or a strong energy bound follows after integration in time. 
\end{proof}
 \textcolor{red}{As mentioned in the introduction, the entropy stability theory has often been applied to nonlinear hyperbolic systems in order to stabilise the related schemes \cite{dalcin2019conservative,dubois1988,hindenlang2019,parsani2015entropy,svard2012,svard2014,svard2021entropy,svar2018,svard2021,chan2022,Gjesteland2022}. In this paper, we aim for a provably stable nonlinear boundary procedure, and focus on smooth solutions, but as in the scalar case we notice that the skew-symmetric form of (\ref{eq:nonlin}) also allows for a mathematical (or generalised) entropy conservation law. 
 Note that (as for the scalar case in Section \ref{sec:theory_il}) the mathematical entropy is not necessarily related to the thermodynamic one \cite{Tadmor2003,Leveque,Godlewski} and that the matrix $P$ is not a function of $U$, such that the mathematical entropy below, is a quadratic not a cubic function of $U$}.
\begin{proposition}\label{:Entropy-relation}
The IBVP  (\ref{eq:nonlin}) together with conditions (\ref{eq:boundcond}) leads to the entropy conservation law
\begin{equation}\label{eq:entropy-equation}
\Phi_t + (\Psi_i)_{x_i} = 0,
\end{equation}
where $\Phi=U^T P U/2$ is the mathematical (or generalised) entropy  and  $\Psi_i=U^T A_i U$ are the entropy fluxes.
\end{proposition}
\begin{proof}
Multiplication of  (\ref{eq:nonlin}) from the left with $U^T$ in a smooth region of the domain yields
\begin{equation}\label{eq:entropy-der}
 (U^T P U/2)_t  +  (U^T A_i U)_{x_i}  = (U_{x_i}^T  A_i U - U^T A_i^T U_{x_i}) - U^T  C U.
\end{equation}
The right-hand side of (\ref{eq:entropy-der}) is cancelled by the skew-symmetry of the equation and (\ref{eq:boundcond}) leading to the entropy conservation relation (\ref{eq:entropy-equation}).
\end{proof}
The entropy conservation law (\ref{eq:entropy-equation}) holds for smooth solutions. In general, discontinuous solutions can develop for non-linear hyperbolic systems,
regardless of their initial smoothness. In the presence of discontinuities, the mathematical entropy conservation law (\ref{eq:entropy-equation}) becomes the entropy inequality
\begin{equation}\label{eq:entropy-equation_ineq}
\Phi_t + (\Psi_i)_{x_i} \leq 0.
\end{equation}
If the mathematical entropy is convex (as in this case where $\Phi_{UU}=P$), then the physically relevant weak solution makes the entropy decrease.
The non-standard compatibility conditions in this case read
\begin{equation}\label{eq:entropy-der}
\partial \Phi/\partial U=\Phi_{U}=U^T P,  \quad \Phi_{U} \lbrack P^{-1}( (A_i U)_{x_i}+A^T_i U_{x_i}+CU)\rbrack=(U^T A_i U)_{x_i}=(\Psi_i)_{x_i}.
\end{equation}
The total entropy in the domain is identical to the energy, similar to what was found for the scalar case in Section \ref{sec:theory_il} above and in \cite{MR4132906, nordstrom2021linear}. 
We will use energy to denote both quantities, but sometimes remind the reader by using both notations explicitly.


\subsection{Energy and entropy rates when $\epsilon > 0$}\label{sec:theory_viscous}
The additional terms that we consider stem from the righthand side of (\ref{eq:nonlin}). Greens formula leads to
\begin{equation}\label{1Dprimalstab_trans}
\epsilon \int\limits_{\Omega}U^T  (D_i)_{x_i} d \Omega  =  \epsilon \oint\limits_{\partial\Omega} U^T (n_i D_i) ds - \epsilon  \int\limits_{\Omega}U_{x_i} ^T D_i d \Omega.
\end{equation}
To obtain a bound for  $\epsilon >  0$, (\ref{eq:boundcond}) is complemented with a condition on the viscous fluxes which leads to
\begin{equation}\label{eq:boundcond_complement}
C+C^T = 0,  \quad  U_{x_i} ^T D_i   \geq 0,  \quad i=1,2,..,k.
\end{equation}
The viscous boundary terms $U^T (n_i D_i)$ can be grouped together with the inviscid ones 
on the left hand side in the energy rate. We summarise the result in the following proposition.
\begin{proposition}\label{lemma:Matrixrelation_complemented}
The IBVP  (\ref{eq:nonlin}) with $\epsilon >  0$ is energy bounded if (\ref{eq:boundcond_complement}) holds
and the boundary conditions (\ref{eq:nonlin_BC}) for $G=0$ lead to
\begin{equation}\label{1Dprimalstab_complemented}
\oint\limits_{\partial\Omega}U^T  (n_i A_i) U -   \epsilon U^T (n_i D_i)  ds \geq  0.
\end{equation}
It is strongly energy bounded if it is energy bounded and the boundary conditions (\ref{eq:nonlin_BC}) for $G \neq 0$ lead to
\begin{equation}\label{1Dprimalstab_strong_complemented}
\oint\limits_{\partial\Omega}U^T  (n_i A_i) U -   \epsilon U^T (n_i D_i)  ds \geq   - \oint\limits_{\partial\Omega} G^TG \\\ ds,
\end{equation}
where $G=G(\vec x,t)$ is independent of the solution $U$.
\end{proposition}
\begin{proof}
The energy method applied to (\ref{eq:nonlin}) with $\epsilon  >  0$ together with (\ref{eq:boundcond_complement}--\ref{1Dprimalstab_strong_complemented}) proves the claim.
\end{proof}

 \section{Part 2: The nonlinear boundary conditions and \textcolor{red}{weak}  implementation procedure}\label{BC_theory} 
\textcolor{red}{We aim for a weak nonlinear boundary procedure that limits the boundary term in terms of given data, as in the scalar case in Section \ref{BC_theory_il}. Similar to the scalar case discussed in Section \ref{numerics_il},  we will thereafter mimic the weak continuous boundary procedure numerically by using the weak SBP-SAT procedure that lead to provable stability. In order to derive the weak continuous boundary procedure, we also employ strong boundary conditions, but only as a mean to arrive at the weak procedure to be mimicked numerically.}  

We start with  the case $\epsilon  \rightarrow  0$. Consider the boundary term
\begin{equation}\label{1Dprimalstab_trans}
\oint\limits_{\partial\Omega}U^T  (n_i A_i)   \\\ U \\\ ds = \frac{1}{2}\oint\limits_{\partial\Omega} U^T ((n_i A_i)  +(n_i A_i )^T) U \\\ ds = \oint\limits_{\partial\Omega}U^T   \tilde A(V)   \\\ U \\\ ds,
\end{equation}
where $\tilde A(V)$ is symmetric. 
Next we transform the matrix $\tilde A$ to diagonal form as $ T^T \tilde A  T =  \Lambda = diag( \lambda_i)$
which gives us new transformed variables $W = T^{-1} U$ 
and
\begin{equation}\label{1Dprimalstab_trans_final}
\oint\limits_{\partial\Omega}U^T   \tilde A(V)   \\\ U \\\ ds  \\\ ds = \oint\limits_{\partial\Omega}W^T   \Lambda   \\\ W\\\ ds = \oint\limits_{\partial\Omega}(W^+)^T   \Lambda^+   \\\ W^+ + (W^-)^T   \Lambda^-   \\\ W^-\\\ ds=\oint\limits_{\partial\Omega}  \lambda_i W_i^2 \\\ ds,
\end{equation}
where we again use Einsteins summation convention. In (\ref{1Dprimalstab_trans_final}) we use the indicator matrices  $I^-, I^+$ where $I^-+I^+=I$ to define, $\Lambda^+ = I^+ \Lambda$ and  $\Lambda^-=  I^- \Lambda$ which are the positive and negative parts of $\Lambda$ respectively while $W^+=I^+ W$  and $W^- = I^- W$ are the corresponding variables. The new transformed variables $W=W(U)$ are functions of the solution in both the linear and nonlinear case. In the nonlinear case,  the diagonal matrix  $\Lambda(U)$ is solution dependent and not a priori bounded while in the linear case, $\Lambda(V)$ is bounded by external data. This difference leads to notable differences in the boundary condition procedure.
\begin{remark}\label{Sylvester}
For linear problems, the number of boundary conditions is equal to the number of eigenvalues of $\tilde A(V)$ with the wrong (in this case negative) sign  \cite {nordstrom2020}. Sylvester's Criterion \cite{horn2012}, shows that the number of boundary conditions is equal to the number of $\lambda_i(V)$  with the wrong sign if the rotation matrix  $T$ is non-singular.  In the nonlinear case where $\lambda_i=\lambda_i(U)$ it is more complicated since multiple forms of the boundary term $W^T   \Lambda W$ may exist,  see Section \ref{Eulerex} below and \cite{nordstrom2022linear-nonlinear,Nordstrom2022_Skew_Euler,nordstrom2021linear} for examples. The nonlinear procedure developed here has similarities with the linear characteristic boundary procedure \cite{kreiss1970,MR436612}. Hence, with a slight abuse of notation we will sometimes refer to $\Lambda(U)$ as "eigenvalues" and to the transformed variables $W(U)$ as "characteristic" variables, although strictly speaking they are not.
\end{remark}

We will impose the boundary conditions both strongly and weakly. For the weak imposition we introduce a lifting operator $L_C$ \textcolor{red}{(similar to the scalar case above)}
 that enforce the boundary conditions in our governing equation (\ref{eq:nonlin}) for $\epsilon  \rightarrow  0$ as 
\begin{equation}\label{eq:nonlin_lif}
P U_t + (A_i(V) U)_{x_i}+A^T_i(V)U_{x_i}+C(V)U+L_C(U)=0,  \quad t \geq 0,  \quad  \vec x=(x_1,x_2,..,x_k) \in \Omega.
\end{equation}
Similar to the scalar case, the lifting operator for two smooth vector functions  $\phi, \psi$ satisfies $\int\limits \phi^T   L_C(\psi) d \Omega = \oint\limits \phi^T  \psi ds$
and enables development  of the numerical boundary procedure in the continuous setting \cite{Arnold20011749,nordstrom_roadmap}.

\subsection{One significant difference between linear and nonlinear boundary conditions}\label{summary_BC_theory}
Before attacking the nonlinear problem, we digress momentarily to the linear case to further stress one aspect of our subsequent nonlinear analysis. 
In the simplest possible version of (\ref{Gen_BC_form}) one can specify $W^-=g$ corresponding to negative $\lambda_i(V)$ indicated by $\lambda_i^-=-|\lambda_i(V)|$. Since $|\lambda_i(V)|$ are bounded, we obtain
\begin{equation}\label{1Dprimalstab_trans_final_extra}
\oint\limits_{\partial\Omega}U^T   \tilde A(V)   \\\ U \\\ ds  \\\  = \oint\limits_{\partial\Omega}W^T   \Lambda   \\\ W\\\ ds =  \oint\limits_{\partial\Omega}  (W^+)^T   \Lambda^+   \\\ W^+ + g^T   \Lambda^-   \\\ g\\\ ds \geq  - \oint\limits_{\partial\Omega}  G^TG \\\ ds,
\end{equation}
where $G_i = \sqrt{|\lambda_i^-(V)|}g_i$. Hence we get a strong energy bound  in terms of external data. However, in the nonlinear case, no estimate is obtained since $\lambda_i^-(U)$ is not a priori bounded, see  Section \ref{BC_theory_il} for the corresponding scalar situation, and also \cite{nordstrom2019, nordstrom2020spatial} for IEE examples

\subsection{The general form of nonlinear boundary conditions for transformed variables}\label{nonlinear_BC_char}
The starting point for the derivation of stable general nonlinear (and linear)  boundary conditions (\ref{eq:nonlin_BC}) is the form (\ref{1Dprimalstab_trans_final}) of the boundary term. First we find the formulation (\ref{1Dprimalstab_trans_final}) with a {\it minimal} number of entries in $\Lambda^-$ \cite{nordstrom2022linear-nonlinear,Nordstrom2022_Skew_Euler,nordstrom2021linear} (there might be more than one formulation of the cubic boundary terms). Next, we specify the transformed characteristic variables $W^-$ in terms of $W^+$ and external data. Finally we need a way to combine different ingoing characteristic variables $W^-$ and a scaling possibility. The formulation 
\begin{equation}\label{Gen_BC_form}
 S^{-1}(\sqrt{|\Lambda^-|}W^--R \sqrt{\Lambda^+}W^+) =G  \quad  \text{or equivalently} \quad  \sqrt{|\Lambda^-|}W^-=R \sqrt{\Lambda^+}W^+ + SG
\end{equation}
where $S^{-1}$ is a non-singular matrix combining values of $W^-$, $S^{-1}R$ combines values of $W^+$ while $G$ is external data will be shown to leads to a bound. 
The boundary condition (\ref{Gen_BC_form}) where we used the notation $ |\Lambda|=diag( |\lambda_i|)$ and $ \sqrt{|\Lambda|} =diag(  \sqrt{|\lambda_i|})$ is general in the sense that it may involve all components of $W$ combined in an arbitrary way by the matrices $S$ and $R$. Also, we observe that the data $G$ must  \textcolor{blue}{represent} some nonlinear interaction since the boundary terms coming from the equations after applying the energy method are cubic, not quadratic \textcolor{red}{(as discussed in Section \ref{sec:theory_il} for the scalar case).}.  A weak implementation of (\ref{Gen_BC_form}) using a lifting operator is given by
\begin{equation}\label{Pen_term_Gen_BC_form}
L_C(U)= L_C(2( I^-T^{-1})^T \Sigma ( \sqrt{|\Lambda^-|}W^--R \sqrt{\Lambda^+}W^+ - SG )),
\end{equation}
where
$\Sigma$ is a penalty matrix. 

The procedure to arrive at  a general stable nonlinear inhomogeneous boundary condition and implementation consists of the following steps for the determination of the unknowns $R,S,\Sigma$ in (\ref{Gen_BC_form}) and (\ref{Pen_term_Gen_BC_form}).
\begin{enumerate}

\item Derive strong homogeneous ($G=0$) boundary conditions. This leads to conditions on matrix $R$.

\item Derive strong inhomogeneous ($G\neq 0$) boundary conditions.  This leads to conditions on matrix $S$.

\item Derive weak homogeneous ($G=0$) boundary conditions. This leads to the specification of matrix $\Sigma$.

\item Show that the weak inhomogeneous ($G\neq 0$) case of the  boundary conditions follows from 1-3 above.

\end{enumerate}
The following Lemma (structured as the step-by-step procedure 1-4 above) is the main result of this paper. 
\begin{lemma}\label{lemma:GenBC}
Consider the boundary term described in (\ref{1Dprimalstab_trans}),(\ref{1Dprimalstab_trans_final}) and the boundary conditions (\ref{Gen_BC_form}) implemented strongly or weakly using (\ref{Pen_term_Gen_BC_form}).

The boundary term augmented with {\bf 1. strong nonlinear homogeneous boundary conditions} is positive semi-definite if the matrix $R$ is such that
\begin{equation}\label{R_condition}
I- R^T  R  \geq 0.
\end{equation}

The boundary term augmented with {\bf 2. strong nonlinear inhomogeneous boundary conditions} is bounded by the data $G$ if
the matrix  $R$ satisfies (\ref{R_condition}) with strict inequality and the matrix $S$ is such that
\begin{equation}\label{S_condition}
I- S^T  S - (R^T S)^T (I- R^T  R)^{-1} (R^T S)  \geq 0 \quad  \text{where}  \quad (I- R^T  R)^{-1}=\sum^{\infty}_{k=0}(R^T R)^k.
\end{equation}

The boundary term augmented with {\bf 3. weak nonlinear homogeneous boundary conditions} is positive semi-definite if the matrix $R$
satisfies (\ref{R_condition}) and the matrix $\Sigma$ is given by
\begin{equation}\label{Sigma_condition}
\Sigma= \sqrt{|\Lambda^-|}.
\end{equation}

The boundary term augmented with {\bf 4. weak nonlinear inhomogeneous boundary conditions} is bounded by the data $G$ if
the matrix $R$ satisfies (\ref{R_condition}) with strict inequality,  the matrix $S$ satisfies  (\ref{S_condition}) and the matrix $\Sigma$ is given by (\ref{Sigma_condition}).
\end{lemma}
 \textcolor{blue}{\begin{remark}\label{explain_notation_nonstadard} 
With a slight abuse of notation we have used t$A \geq 0$ to indicate that the matrix $A$ is positive semi-definite and also written that if the inequality holds strictly such that  $A > 0$ , it indicates positive definiteness.
We have used this notation in (\ref{R_condition}) and (\ref{S_condition}) and will continue to use it below.
\end{remark}}

\begin{proof} We proceed in the step-by-step manner described above.

{\it 1. The homogeneous boundary condition (\ref{Gen_BC_form}) implemented strongly} (with $G=0$) lead to 
\begin{equation}\label{R1_derivation}
W^T   \Lambda   W  = (W^+)^T \Lambda^+ W^+ - (W^-)^T |\Lambda^-| W^- =  (\sqrt{\Lambda^+}W^+)^T (I-R^T R)  (\sqrt{\Lambda^+}W^+)  \geq 0
\end{equation}
if condition (\ref{R_condition}) holds. 

{\it 2. The inhomogeneous boundary condition (\ref{Gen_BC_form}) implemented strongly} (with $G\neq0$)  lead to
\begin{equation}\label{S1_derivation}
 W^T   \Lambda   W  = (\sqrt{\Lambda^+}W^+)^T (\sqrt{\Lambda^+}W^+) -  (R (\sqrt{\Lambda^+}W^+) + S G)^T (R (\sqrt{\Lambda^+}W^+) + S G).
\end{equation}
Expanding (\ref{S1_derivation}), adding and subtracting $G^T G$ lead to the result
\begin{equation}
\label{estimate_2}
 W^T   \Lambda   W =
\begin{bmatrix}
\sqrt{\Lambda^+} W^+ \\
G
\end{bmatrix}^T
\begin{bmatrix}
I- R^T R & - R^T S \\
-S^T R  &  I-S^T S
\end{bmatrix}
\begin{bmatrix}
\sqrt{\Lambda^+} W^+ \\
G
\end{bmatrix} - G^T G. 
\end{equation}
The matrix in (\ref{estimate_2}) can be rotated into block-diagonal form with the upper left block preserved and the lower right block being
$I- S^T  S - (R^T S)^T\left( I- R^T R \right)^{-1}  (R^T S)$ since (\ref{R_condition}) \textcolor{green}{holds strictly. Next we choose $S$ such that (\ref{S_condition}) holds which} guarantees that the left term in (\ref{estimate_2}) is non-negative.

{\it 3. The homogeneous boundary condition (\ref{Gen_BC_form}) implemented weakly} (with $G=0$)  using the lifting operator in (\ref{Pen_term_Gen_BC_form}) and applying the energy method lead to the boundary term
\begin{equation}\label{Sigma1_derivation}
 W^T \Lambda   W  + 2 U^T  ( I^-T^{-1})^T \Sigma (\sqrt{|\Lambda^-|}W^- - R \sqrt{\Lambda^+}W^+))= W^T   \Lambda   W + 2 (W^-)^T  \Sigma (\sqrt{|\Lambda^-|}W^- - R \sqrt{\Lambda^+}W^+). \nonumber
 \end{equation}
 Collecting similar terms and making the choice (\ref{Sigma_condition}) for  $\Sigma$ transforms the right hand side to
\begin{equation}\label{Sigma2_derivation}
 (\sqrt{\Lambda^+}W^+)^T  (\sqrt{\Lambda^+}W^+) +( \sqrt{|\Lambda^-|}W^-)^T ( \sqrt{|\Lambda^-|}W^-) -2 ( \sqrt{|\Lambda^-|}W^-)^T R (\sqrt{\Lambda^+}W^+).
\end{equation}
Adding and subtracting  $(R \sqrt{\Lambda^+}W^+)^T  (R \sqrt{\Lambda^+}W^+)$ and using (\ref{R_condition}) guarantee positivity of (\ref{Sigma2_derivation}) since
\begin{equation}\label{Sigma4_derivation}
  (\sqrt{\Lambda^+}W^+)^T (I - R^T R)(\sqrt{\Lambda^+}W^+) + (\sqrt{|\Lambda^-|}W^--R \sqrt{\Lambda^+} W^+)^T (\sqrt{|\Lambda^-|}W^--R \sqrt{\Lambda^+} W^+) \geq 0.
\end{equation}

{\it 4. The inhomogeneous boundary condition (\ref{Gen_BC_form}) implemented weakly} (with $G\neq0$)  using the lifting operator in (\ref{Pen_term_Gen_BC_form}) and the choice $\Sigma$ in (\ref{Sigma_condition}) lead to the boundary terms
\begin{equation}\label{Sigma2_derivation_new}
W^T \Lambda   W   + 2 (\sqrt{|\Lambda^-|}W^-)^T  (\sqrt{|\Lambda^-|} W^- - R \sqrt{\Lambda^+}W^+  -  SG)).
\end{equation}
By adding and subtracting $G^T G $  and rearranging, the boundary terms (\ref{Sigma2_derivation_new}) above can be written as
\begin{align}
& (\sqrt{|\Lambda^-|}W^- - R  \sqrt{\Lambda^+}W^+ - SG)^T  (\sqrt{|\Lambda^-|}W^- - R  \sqrt{\Lambda^+}W^+ - SG) + \nonumber \\ 
&  \begin{bmatrix}
\sqrt{\Lambda^+} W^+ \\
G
\end{bmatrix}^T
\begin{bmatrix}
I- R^T R & - R^T S \\
-S^T R  &  I-S^T S
\end{bmatrix}
\begin{bmatrix}
\sqrt{\Lambda^+} W^+ \\
G
\end{bmatrix} - G^TG \label{generic_discrete_BT_lifting_details_2}
\end{align}
The first term is obviously positive semi-definite and the remaining part is identical to the term in (\ref{estimate_2}).
\end{proof}
\begin{remark}
Lemma \ref{lemma:GenBC} can be used to prove that the estimates (\ref{1Dprimalstab}) and (\ref{1Dprimalstab_strong}) in Proposition \ref{lemma:Matrixrelation} holds.
\end{remark}

We are now ready to connect the characteristic boundary condition formulation (\ref{Gen_BC_form}) with (\ref{eq:nonlin_BC}) in the original variables.  The definitions $W=T^{-1}U$, $W^-=I^-W$, $W^+=I^+W$, $\sqrt{|\Lambda^-|}=I^- \sqrt{|\Lambda|}$ and $\sqrt{\Lambda^+}=I^+ \sqrt{|\Lambda|}$ transforms  (\ref{Gen_BC_form})  to
\begin{equation}\label{Gen_BC_form_expanded}
S^{-1} (I^- - R I^+) \sqrt{|\Lambda|}T^{-1}U =G  \quad  \text{and hence} \quad  B=S^{-1} (I^- - R I^+) \sqrt{|\Lambda|}T^{-1}.
\end{equation}
This concludes the analysis of the formulation of nonlinear boundary conditions for first derivative terms.
\textcolor{red}{\begin{remark}
The new boundary procedure presented above generalise the well known linear characteristic boundary procedure  \cite{kreiss1970,MR436612} by inserting the additional scaling with the solution dependent eigenvalues. We show below that the new nonlinear boundary procedure can also be used to bound IBVPs involving second derivatives similar to what has been done in the linear case \cite{nordstrom2005,MR1339182,MR1669660}.
\end{remark}}

\subsection{Extension of the analysis to include second derivative terms}\label{nonlinear_BC_visc}
The boundary terms to consider in the case when $\epsilon >  0$  are given in Proposition \ref{lemma:Matrixrelation_complemented}. The argument in the surface integral
can be reformulated into the first derivative setting.  By denoting the symmetric part of $n_i A_i$ as $\tilde A$ and introducing the notation $n_i D_i/2=\tilde F$ for the viscous flux we rewrite the boundary terms as
\begin{equation}\label{newflux_form_1}
U^T  (n_i A_i) U -   \epsilon U^T (n_i D_i) =  U^T \tilde A U -  \epsilon U^T  \tilde F - \epsilon \tilde F^T  U=
\begin{bmatrix}
U \\
\epsilon \tilde F
\end{bmatrix}^T
\begin{bmatrix}
\tilde A & - I \\
 - I &  0
\end{bmatrix}
\begin{bmatrix}
U \\
 \epsilon \tilde F
\end{bmatrix}
\end{equation}
where $I$ is the identity matrix. We can now formally diagonalise the boundary terms in (\ref{newflux_form_1}), apply the boundary condition (\ref{Gen_BC_form}) and  Lemma \ref{lemma:GenBC} to obtain energy bounds.

However, it is instructive the take the analysis one step further. We start by transforming the boundary terms (\ref{newflux_form_1})  to block-diagonal  form (assuming that $\tilde A$ is non-singular)
\begin{equation}\label{newflux_form_2}
\begin{bmatrix}
U \\
\epsilon \tilde F
\end{bmatrix}^T
\begin{bmatrix}
\tilde A & -I \\
 - I &  0
\end{bmatrix}
\begin{bmatrix}
U \\
 \epsilon \tilde F
\end{bmatrix}
=
 \begin{bmatrix}
U - \epsilon  \tilde A^{-1} \tilde F\\
     \epsilon  \tilde F
\end{bmatrix}^T
\begin{bmatrix}
\tilde A & 0 \\
          0 & -  \tilde A^{-1} 
\end{bmatrix}
\begin{bmatrix}
U - \epsilon  \tilde A^{-1}  \tilde F\\
 \epsilon \tilde F
\end{bmatrix}.
\end{equation}
Next, we insert the previous transformations  $\tilde A=(T^{-1})^T  \Lambda T^{-1}$ and $W=T^{-1} U$ into (\ref{newflux_form_2}) to get
\begin{equation}\label{matrix-vector_1}
 \begin{bmatrix}
 W - \epsilon  \Lambda^{-1}  T^T  \tilde F\\
                 \epsilon  T^T  \tilde F
\end{bmatrix}^T
\begin{bmatrix}
\Lambda & 0 \\
          0 & -   \Lambda^{-1}
\end{bmatrix}
\begin{bmatrix}
W - \epsilon  \Lambda^{-1}  T^T  \tilde F \\
                \epsilon  T^T  \tilde F
\end{bmatrix}.
\end{equation}
Finally we move the factor $ \Lambda^{-1}$ from the matrix into the dependent vectors as shown below
\begin{equation}\label{matrix-vector_False}
\tilde W^T \tilde  \Lambda \tilde W
=
 \begin{bmatrix}
W  - \epsilon  \Lambda^{-1}   T^T  \tilde F\\
              \,\ \quad - \epsilon   \Lambda^{-1}  T^T  \tilde F
\end{bmatrix}^T
\begin{bmatrix}
\Lambda & 0 \\
          0 & -   \Lambda
\end{bmatrix}
\begin{bmatrix}
W - \epsilon  \Lambda^{-1}   T^T  \tilde F \\
             \,\ \quad  - \epsilon \Lambda^{-1}  T^T  \tilde F
\end{bmatrix},
\end{equation}
which results in the same formal setting as for the first derivative problems. 
\begin{remark}
If there are no hidden linear dependencies in (\ref{matrix-vector_False}), then we can directly apply the boundary condition (\ref{Gen_BC_form}) and  Lemma \ref{lemma:GenBC} to obtain energy bounds.  However, for incompletely parabolic equations, the vector $\tilde F$ is shorter than $W$ and reduce the number of boundary conditions below the number indicated by the diagonal entries in $\tilde \Lambda$. We will discuss this phenomenon in detail for the INSEs below. 
\end{remark}

\section{Part 3: Semi-discrete nonlinear stability}\label{numerics}
As in the scalar case in Section \ref{numerics_il}, we show how the derived boundary conditions and penalty parameters that lead to a nonlinear energy bound also lead to nonlinear energy stability in the nonlinear system case.
Consider the extended version (\ref{eq:nonlin_lif}) of (\ref{eq:nonlin}) rewritten (with Einsteins summation convention) here for clarity
\begin{equation}\label{eq:stab_eq}
P U_t + (A_i U)_{x_i}+A^T_i U_{x_i}+CU+L_C=0,  \quad t \geq 0,  \quad  \vec x=(x_1,x_2,..,x_k) \in \Omega.
\end{equation}
 Equation (\ref{eq:stab_eq}) is augmented with the initial condition $U(\vec x,0)=F(\vec x)$ in $\Omega$ and boundary conditions of the form (\ref{Gen_BC_form}) on $\partial\Omega$.  Furthermore $A_i=A_i(U)$, $C=C(U)$  and $P$ are  $n \times n$ matrices while $U$ and $L_C$ are $n$ vectors.  $L_C$ is the continuous lifting operator of the form (\ref{Pen_term_Gen_BC_form}) implementing the boundary conditions weakly. 
 A straightforward approximation of (\ref{eq:stab_eq}) on SBP-SAT form in $M$ nodes is
\begin{equation}\label{EUL_Disc}
(P \otimes I_M) \vec U_t+{\bf D_{x_i}} {\bf A_i}  \vec U+{\bf A_i^T} {\bf D_{x_i}} \vec U   + {\bf C}  \vec U + {\vec L_D}=0, \quad \vec U(0) = \vec F
\end{equation}
where $\vec U=(\vec U_1^T, \vec U_2^T,...,\vec U_n^T)^T$ include approximations of  $U=(U_1,U_2,...,U_n)^T$ in each node. The discrete lifting operator ${\vec L_D} (\vec U)$ implements the boundary conditions in a similar way to  $L_C(U)$ and $\vec F$ denotes the discrete initial data with the continuous initial data injected in the nodes. The matrix elements of ${\bf A_i},{\bf C}$ are matrices with node values of the matrix elements in $A_i,C$ injected on the diagonals as exemplified below
\begin{equation}
\label{illustration}
A_i=
\begin{bmatrix}
      a_{11}   &  \ldots  & a_{1n} \\
       \vdots   & \ddots & \vdots \\
       a_{n1} &  \ldots  & a_{nn}
\end{bmatrix}, \quad
{\bf A_i} =
\begin{bmatrix}
      {\bf a_{11}}   &  \ldots  &  {\bf a_{1n}}  \\
       \vdots          & \ddots &  \vdots           \\
        {\bf a_{n1}} &  \ldots &  {\bf a_{nn}} 
\end{bmatrix}, \quad
{\bf a_{ij}} =diag(a_{ij}(x_1,y_1), \ldots, a_{ij}(x_M,y_M)).
\end{equation}
Moreover ${\bf D_{x_i}}=I_n \otimes D_{x_i}$ where $\otimes$ denotes the Chronicler product, $I_n$ is the  $n \times n$ identity matrix, $D_{x_i}=P_{\Omega}^{-1} Q_{x_i}$ are SBP difference operators, $P_{\Omega}$ is a positive definite diagonal volume quadrature matrix that defines a scalar product and norm such that  
\begin{equation}
\label{volumenorm}
(\vec U, \vec V)_{\Omega} = \vec U^TP_{\Omega}  \vec V \approx  \int\limits_{\Omega}U^T V d \Omega, \quad \text{and}  \quad  (\vec U, \vec U)_{\Omega} = \| \vec U\|^2_{\Omega}  = \vec U^TP_{\Omega}  \vec U \approx  \int\limits_{\Omega}U^T U d \Omega = \| U\|^2_{\Omega} .
\end{equation}

Following \cite{Lundquist201849} we introduce the discrete normal $\mathbf{N}=(N_1,N_2,...,N_k)$ approximating the continuous normal $\mathbf{n}=(n_1,n_2,...,n_k)$ in the $N$ boundary nodes and a restriction operator $E$ that extracts the boundary values $E\vec U$ from the total values. We also need a positive definite diagonal boundary quadrature $P_{\partial\Omega}=diag(ds_1,ds_2,...,ds_N)$ such that $\oint_{\partial \Omega}  U^T U d.s. \approx (E \vec U)^T P_{\partial\Omega} (\vec U)= (EU)_i^2 ds_i$. With this notation in place (again using Einsteins summation convention), the SBP constraints for a scalar variable becomes
\begin{equation}
\label{SBP_constraint}
 Q_{x_i}+Q_{x_i}^T=E^T P_{\partial\Omega} N_{i} E, 
\end{equation}
which leads to the scalar summation-by-parts formula mimicking integration-by-parts
\begin{equation}
\label{SBP_scalar_relation}
(\vec U,  D_{x_i}  \vec V) = \vec U^T  P_{\Omega} (D_{x_i}  \vec V) = - ( D_{x_i}  \vec U, \vec V)  + (E \vec U)^T P_{\partial\Omega} N_{i} (E \vec V).
\end{equation}
The scalar SBP relations in (\ref{SBP_constraint}),(\ref{SBP_scalar_relation}), correspond to the SBP formulas for a vector with $n$ variables as
\begin{equation}\label{Multi-SBP}
(\vec U, {\bf D_{x_i}} \vec V)=\vec U^T  (I_n   \otimes P_{\Omega})({\bf D_{x_i}} \vec V)= -({\bf D_{x_i}} \vec U, \vec V) + (E\vec U)^T  (I_n   \otimes P_{\partial\Omega}) N_{i}  (E\vec V).
\end{equation}

It remains to construct the discrete lifting operator $\vec L_D$ (often called the SAT term  \cite{svard2014review,fernandez2014review}) such that we can reuse the continuous analysis. We consider an operator of the form $\vec L_D =(I_n \otimes P_{\Omega})(DC) \vec L_C$.
The transformation matrix $DC$ first extracts the boundary nodes from the volume nodes, secondly permute the dependent variables from being organised as $(E\vec U_1, E\vec U_2,..,E\vec U_n)^T$  to $((E \vec U)_1, (E \vec U)_2,.., (E \vec U)_N)^T$ using the permutation matrix $P_{erm}$ 
and thirdly numerically integrate the resulting vector against the continuous lifting operator $ \vec L_C$ (now applied to the discrete solution).  More specifically we have
\begin{align}
&\vec L_D=(I_n \otimes P_{\Omega})^{-1})(DC) \vec L_C&  &\ \ \ DC=(I_n \otimes E^T) (P_{erm}) ^T (P_{\partial\Omega} \otimes I_n),& \label{generic_discrete_BT_lifting_details_1} \\ 
&\vec L_C=((L_C)_1,(L_C)_2,...,(L_C)_N)^T&             &(L_C)_j=(2 (J^- T^{-1})^T  \Sigma (\sqrt{|\Lambda^-|} \vec W^--R \sqrt{\Lambda^+} \vec W^+ - S \vec G))_j.& \label{generic_discrete_BT_lifting_details_2}
\end{align}

We can now prove the semi-discrete correspondence to Proposition \ref{lemma:Matrixrelation}.
\begin{proposition}\label{lemma:Matrixrelation_discrete}
Consider the nonlinear scheme (\ref{EUL_Disc}) with $\vec L_D $ defined in (\ref{generic_discrete_BT_lifting_details_1}) and  (\ref{generic_discrete_BT_lifting_details_2}).

 It is nonlinearly stable for $\vec G=0$ if the relations (\ref{R_condition}) and (\ref{Sigma_condition}) in Lemma \ref{lemma:GenBC} hold and the solution satisfies the estimate 
\begin{equation}\label{1Dprimalstab_disc}
\| \vec U \|_{P  \otimes P_{\Omega}}^2 \leq \| \vec F \|_{P  \otimes P_{\Omega}}^2.
\end{equation}

It is strongly nonlinearly stable for $\vec G \neq 0$ if the relations (\ref{R_condition}),(\ref{S_condition}) and (\ref{Sigma_condition}) in Lemma \ref{lemma:GenBC} hold and the solution satisfies the estimate 
\begin{equation}\label{1Dprimalstab_strong_disc}
\| \vec U \|_{P  \otimes P_{\Omega}}^2 \leq \| \vec F \|_{P  \otimes P_{\Omega}}^2 +  2 \int_0^t  \sum_{j=1,N} \lbrack \vec G^T \vec G) \rbrack_j ds_j \ dt.
\end{equation}
In (\ref{1Dprimalstab_disc}) and (\ref{1Dprimalstab_strong_disc}), $\vec F$ and $\vec G$ are external data from $F$ and $G$ injected in the nodes. 
\end{proposition}
\begin{proof}
The discrete energy method (multiply (\ref{EUL_Disc}) from the left with  $\vec U^T (I_n \otimes P_{\Omega}$)) yields 
\begin{equation}\label{Disc_energy_initial}
\vec U^T (P  \otimes P_{\Omega})  \vec U_t+ (\vec U,  {\bf D_{x_i}} {\bf A_i}  \vec U)+  ({\bf A_i} \vec U, {\bf D_{x_i}} \vec U)+(\vec U,{\vec L_D)}=0,                                                       
\end{equation}
where we have used that $(I_n \otimes P_{\Omega})$ commutes with ${\bf A_i}$ (since the matrices have diagonal blocks) and that the symmetric part of  $C$ is zero. The SBP constraints (\ref{Multi-SBP}) and the notation $\vec U^T (P  \otimes P_{\Omega})  \vec U= \| \vec U \|_{P  \otimes P_{\Omega}}^2 $ simplifies (\ref{Disc_energy_initial}) to
\begin{equation}\label{Disc_energy_final}
\dfrac{1}{2} \dfrac{d}{dt} \| \vec U \|_{P  \otimes P_{\Omega}}^2 
+ \vec U^T (I_n  \otimes E^T P_{\partial\Omega} N_{i} E )  {\bf A_i} \vec U  + (\vec U,{\vec L_D}) =0.
\end{equation}

The semi-discrete energy rate in (\ref{Disc_energy_final}) mimics the continuous energy rate in the sense that only boundary terms remain. To make use of the already performed continuous energy analysis, we expand the boundary terms and exploit the diagonal form of  $P_{\partial\Omega}$ and the relation $W_i = (T^{-1}(E \vec U))_i $. The result is
\begin{equation}\label{Disc_energy_final_BT}
\vec U^T (I_n  \otimes E^T P_{\partial\Omega} N_{x_i} E )  {\bf A_i} \vec U = \sum_{j=1,N} \lbrack  (E  \vec U)^T (N_{i}  {\bf A_i})  (E  \vec U) \rbrack_j ds_j = \sum_{j=1,N}  \lbrack \vec W^T   \Lambda   \vec W \rbrack_jds_j.
\end{equation}
The  discrete boundary terms in (\ref{Disc_energy_final_BT}) now have the same form as the continuous ones in (\ref{1Dprimalstab_trans_final}). By using  (\ref{generic_discrete_BT_lifting_details_1}) and (\ref{generic_discrete_BT_lifting_details_2}) we find that
\begin{equation}\label{discrete_BT_lifting}
(\vec U,{\vec L_D})= \sum_{j=1,N}  \lbrack 2  (\vec W^-)^T  \Sigma(\sqrt{|\Lambda^-|} \vec W^- - R \sqrt{\Lambda^+} \vec W^+ - S \vec G) \rbrack_j ds_j.
\end{equation}
The combination of  (\ref{Disc_energy_final})-(\ref{discrete_BT_lifting}) leads to the final form of the energy rate
\begin{equation}\label{Disc_energy_final_final}
\dfrac{1}{2} \dfrac{d}{dt} \| \vec U \|_{P  \otimes P_{\Omega}}^2 + \sum_{j=1,N}  \lbrack  \vec W^T   \Lambda \vec W + 2  (\vec W^-)^T   \Sigma(\sqrt{|\Lambda^-|} \vec W^- - R \sqrt{\Lambda^+} \vec W^+ - S \vec G) \rbrack_j  ds_j=0.
\end{equation}
By using (\ref{R_condition})-(\ref{Sigma_condition}) in Lemma \ref{lemma:GenBC}, the estimates (\ref{1Dprimalstab_disc}) and (\ref{1Dprimalstab_strong_disc}) follow by using the same technique that was used for the continuous estimates in the proof of Lemma \ref{lemma:GenBC}.
\end{proof}

\section{Application of the general theory to initial boundary value problems in CFD}\label{Examples}
In the previous sections we have developed a general boundary condition and implementation procedure which generalise the well-known linear characteristic one in \cite{kreiss1970,MR436612,MR1339182,MR1669660,nordstrom2005}. In this section we exemplify it's use for common IBVPs in CFD. We begin with first derivative examples including the IEEs, the SWEs and the CEEs and conclude with an analysis of a second derivative case exemplified by the INSEs.

\subsection{The 2D incompressible Euler equations}\label{Eulerex}
The incompressible 2D Euler equations in split form are
\begin{equation}
P U_t 
+ \frac{1}{2}\left[ (A U)_x + A U_x 
+  (B U)_y + B U_y \right]
= 0.
\label{NS_splitting}
\end{equation} 
where $U=(u,v,p)^T$ and
\begin{equation}
P =
\begin{bmatrix}
1 & 0 & 0 \\
0 & 1 & 0 \\
0 & 0 & 0
\end{bmatrix},
\quad
A=
\begin{bmatrix}
u & 0 & 1 \\
0 & u & 0 \\
1 & 0 & 0
\end{bmatrix},
\quad
B=
\begin{bmatrix}
v & 0 & 0 \\
0 & v & 1 \\
0 & 1 & 0
\end{bmatrix}.
\label{ABI}
\end{equation}
Since the matrices $A,B$ are symmetric, the formulation (\ref{NS_splitting}) is in the required skew-symmetric form (\ref{eq:nonlin_lif})
We obtain an estimate in the semi-norm  $\|U\|^2_P= \int_{\Omega} U^T P  U d \Omega$ involving only the velocities. Note that the pressure $p$ includes a division by the constant density, and hence has the dimension velocity squared.

By applying the transformations  \textcolor{blue}{in the beginning of Section  \ref{BC_theory}}, the boundary term gets the form
\begin{equation}\label{1Dprimalstab_trans_final_icomp}
U^T (n_1 A + n_2 B)U = U^T_n \tilde A U_n=W^T  \Lambda W= (W^+)^T   \Lambda^+   W^+ + (W^-)^T   \Lambda^- W^-
\end{equation}
where the solution vector is rotated into the normal  $(u_n=n_1u+n_2v)$ and tangential ($ u_{\tau}=-n_2u+n_1v)$ direction, and $W = T^{-1}U_n$ such that
\begin{equation}\label{NS_BT_rotated}
U_n=
\begin{bmatrix}
u_n\\
u_{\tau} \\
p
\end{bmatrix},
\tilde A \textcolor{blue}{=}
\begin{bmatrix}
u_n & 0     & 1\\
0      &u_n & 0 \\
1      &0    & 0
\end{bmatrix},
T^{-1}=
\begin{bmatrix}
1 & 0     & 1/u_n\\
0      &1 & 0 \\
0      &0    &1
\end{bmatrix},
W=
\begin{bmatrix}
u_n+p/u_n\\
u_{\tau} \\
p
\end{bmatrix},
\Lambda=
\begin{bmatrix}
u_n & 0     & 0\\
0      &u_n & 0 \\
0      &0    & -1/u_n
\end{bmatrix}.
\end{equation}

 At inflow where $u_n < 0$ 
\begin{equation}
W^-=
\begin{bmatrix}
u_n+p/u_n\\
u_{\tau} 
\end{bmatrix},
\quad
\Lambda^- =
\begin{bmatrix}
u_n & 0 \\
0 & u_n 
\end{bmatrix},
\quad
W^+= p,
\quad
\Lambda^+ = -1/u_n
\label{IEE_inflow}
\end{equation}
while at outflow with $u_n > 0$ we get the reversed situation with
\begin{equation}
W^+=
\begin{bmatrix}
u_n+p/u_n\\
u_{\tau} 
\end{bmatrix},
\quad
\Lambda^+ =
\begin{bmatrix}
u_n & 0 \\
0 & u_n 
\end{bmatrix},
\quad
W^-= p,
\quad
\Lambda^- = -1/u_n.
\label{IEE_outflow}
\end{equation}
By using the definitions in (\ref{IEE_inflow}) and (\ref{IEE_outflow}), it is straightforward to check for boundedness using Lemma \ref{lemma:GenBC}.
\begin{example}\label{IEE_EX} Recall that we must allow for some nonlinear scaling or interaction since the boundary terms coming from the equations after applying the energy method are typically cubic \textcolor{red}{as discussed in Section \ref{nonlinear_BC_char}}.

\textcolor{blue}{Consider the general boundary condition (\ref{Gen_BC_form}) for the inflow case (\ref{IEE_inflow} above which reads  
\begin{equation}
S^{-1} \left(
\begin{bmatrix}
\sqrt{|u_n|} & 0 \\
0 & \sqrt{|u_n|} 
\end{bmatrix}
\begin{bmatrix}
u_n+p/u_n\\
u_{\tau} 
\end{bmatrix}
-
\begin{bmatrix}
R_1 \\
R_2
\end{bmatrix} 
p/ \sqrt{|u_n|}  \right) = G.
\label{IEE_inflow_ex_0}
\end{equation}
Let us for the sake of argument assume that we want to impose a Dirichlet condition on the tangential velocity, and investigate what other appropriate boundary condition to use.  The choice $R_2=0$ and $S^{-1}_{21}=0$ removes the pressure form the second equation in 
(\ref{IEE_inflow_ex}) and imposes a scaled Dirichlet condition on $u_{\tau}$.  
\begin{equation}
(S^{-1}_{22}/ \sqrt{|u_n|})(|u_n|u_{\tau}) = G_2.
\label{IEE_inflow_ex_tangentiall}
\end{equation}
The remaining first equation (using a diagonal $S$ matrix) becomes 
\begin{equation}
-(S^{-1}_{11}/ \sqrt{|u_n|})(u_n^2 + (1+R_1) p) = G_1.
\label{IEE_inflow_ex_normal}
\end{equation}
The choice $R_1=-1$ removes the influence of the pressure $p$ and leaves a scaled Dirichlet condition on the normal velocity  $u_n$, while $R_1=1$ leaves a scaled total pressure condition.  Both of these conditions in combination with the Dirichlet condition on the tangential velocity are frequently used in practical calculations.} 

\textcolor{blue}{However, the combination of these conditions (i.e.  $(R_1,R_2)^T=(\pm 1, 0)^T$) lead to $I - R^T R=0$ and hence condition  (\ref{R_condition}) is satisfied, but not strictly. This makes the choice of $S$ in (\ref{IEE_inflow_ex}) (and (\ref{S_condition})) irrelevant since only the homogeneous case is possible. This leads to boundedness, but not strong boundedness as defined in Proposition \ref{lemma:Matrixrelation}. By slightly modifying the choice in the first equation such that $R_1 \neq \pm 1$ we can invert $I - R^T R$ as required in (\ref{S_condition}). Making the ansatz S=$diag(\alpha,\beta)$, condition (\ref{S_condition}) becomes
\begin{equation}
\left(
\begin{matrix}
1- \alpha^2/(1-R_1^2)& 0 \\
0 & 1- \beta^2
\end{matrix} \right) \geq 0 
\label{IEE_inflow_ex_Scompute}
\end{equation}
which implies that $0 < \alpha^2 \leq  1-R_1^2$ and $0 < \beta^2 \leq 1$ and $R_1^2 < 1$ must hold for strong boundeness. This example exemplifies how one can modify "common often used" boundary conditions to get more stable ones. 
A weak implementation require in addition $\Sigma= \sqrt{|\Lambda^- |} =diag( \sqrt{|u_n|}, \sqrt{|u_n|})$ as specified in (\ref{Sigma_condition}).}

For an outflow condition on the pressure, the boundary condition (\ref{Gen_BC_form}) becomes
\begin{equation}
S^{-1} \left(
p/\sqrt{|u_n|}  - (R_1,R_2)
\begin{bmatrix}
\sqrt{|u_n|} & 0 \\
0 & \sqrt{|u_n|} 
\end{bmatrix}
\begin{bmatrix}
u_n+p/u_n\\
u_{\tau} 
\end{bmatrix}
\right) = G.
\label{IEE_outflow_ex}
\end{equation}
The characteristic boundary condition for the pressure $p$ is recovered with $R=(0,0)$
and hence condition (\ref{R_condition})  holds strictly.  This leads to a strongly energy bounded solution for  $S$ such that $I-S^T S \geq 0$ as required in (\ref{S_condition}) for a vanishing  $R$. A weak  implementation require  $\Sigma= \sqrt{|\Lambda^- |} =1/\sqrt{|u_n|}$  as specified in (\ref{Sigma_condition}).
\end{example} 

\begin{example}\label{IEE_EX} Recall that we must allow for some nonlinear scaling or interaction since the boundary terms coming from the equations after applying the energy method are typically cubic \textcolor{red}{as discussed in Section \ref{nonlinear_BC_char}}.

Dirichlet like inflow conditions on the velocities using the boundary condition (\ref{Gen_BC_form}) are obtained by
\begin{equation}
S^{-1} \left(
\begin{bmatrix}
\sqrt{|u_n|} & 0 \\
0 & \sqrt{|u_n|} 
\end{bmatrix}
\begin{bmatrix}
u_n+p/u_n\\
u_{\tau} 
\end{bmatrix}
-
\begin{bmatrix}
R_1 \\
R_2
\end{bmatrix} 
p/ \sqrt{|u_n|}  \right) = G
\label{IEE_inflow_ex}
\end{equation}
where  \textcolor{blue}{$R=(-1,0)^T$} removes the influence of the pressure $p$. This lead to $I - R^T R=0$ and hence condition  (\ref{R_condition}) is satisfied, but not strictly. This makes the choice of $S$ in (\ref{IEE_inflow_ex}) (and (\ref{S_condition})) irrelevant since only the homogeneous case is possible. This leads to boundedness, but not strong boundedness as defined in Proposition \ref{lemma:Matrixrelation}. A weak 
implementation require $\Sigma= \sqrt{|\Lambda^- |} =diag( \sqrt{|u_n|}, \sqrt{|u_n|})$ as specified in (\ref{Sigma_condition}).

For an outflow condition on the pressure, the boundary condition (\ref{Gen_BC_form}) becomes
\begin{equation}
S^{-1} \left(
p/\sqrt{|u_n|}  - (R_1,R_2)
\begin{bmatrix}
\sqrt{|u_n|} & 0 \\
0 & \sqrt{|u_n|} 
\end{bmatrix}
\begin{bmatrix}
u_n+p/u_n\\
u_{\tau} 
\end{bmatrix}
\right) = G.
\label{IEE_outflow_ex}
\end{equation}
The characteristic boundary condition for the pressure $p$ is recovered with $R=(0,0)$
and hence condition (\ref{R_condition})  holds strictly.  This leads to a strongly energy bounded solution for  $S$ such that $I-S^T S \geq 0$ as required in (\ref{S_condition}) for a vanishing  $R$. A weak  implementation require  $\Sigma= \sqrt{|\Lambda^- |} =1/\sqrt{|u_n|}$  as specified in (\ref{Sigma_condition}).
\end{example}

\subsection{The 2D shallow water equations}\label{SWEex_2D}

The 2D SWEs on the skew-symmetric form presented in Proposition \ref{lemma:Matrixrelation} were derived in  \cite{nordstrom2022linear-nonlinear}. They are
\begin{equation}\label{eq:swNoncons_new_skew}
    U_t +  (A U)_x + A^T U_x+ (B U)_y + B^T U_y+CU= 0,
\end{equation}
where $U=(U_1,U_2,U_3)^T=(\phi, \sqrt{\phi} u, \sqrt{\phi} v))^T$, $\phi = g h$ is the geopontential  \cite{oliger1978},  $h$ is the water height, $g$ is the gravitational constant and $(u,v)$ is the fluid velocity in $(x,y)$ direction respectively.  The Coriolis forces are included in the matrix $C$ with the function $f$ which is typically a function of latitude \cite{shallowwaterbook,whitham1974}. Note that $h>0$ and $\phi>0$ from physical considerations.  The matrices in (\ref{eq:swNoncons_new_skew}) constitute a
two-parameter family 
\begin{equation}\label{eq:swNoncons_new_matrix_ansatz_sol_A}
A =        \begin{bmatrix}
     \alpha \frac{U_2}{\sqrt{U_1}}  & (1-3 \alpha) \sqrt{U_1}                  & 0 \\
     2\alpha  \sqrt{U_1}                 & \frac{1}{2}  \frac{U_2}{\sqrt{U_1}} &  0  \\
     0                                            & 0                                                    &\frac{1}{2} \frac{U_2}{\sqrt{U_1}}
       \end{bmatrix},
       B = 
       \begin{bmatrix}
     \beta \frac{U_3}{\sqrt{U_1}}  & 0               &  (1-3 \beta) \sqrt{U_1}    \\
     0              & \frac{1}{2}  \frac{U_3}{\sqrt{U_1}} &  0  \\
     2\beta  \sqrt{U_1}                                         & 0                                                    &\frac{1}{2} \frac{U_3}{\sqrt{U_1}}
       \end{bmatrix}, 
  C = \begin{bmatrix}
     0 & 0  & 0      \\
     0 & 0  &  -f    \\
     0 & +f & 0
       \end{bmatrix}
\end{equation}
where the parameters $\alpha, \beta$ are arbitrary. (Symmetric matrices are e.g. obtained with $\alpha=\beta=1/5$.) 
\label{rem:noalpha_2d}

By computing the boundary term, we find
\begin{equation}\label{boundarmatrix_2D}
U^T (n_1A+n_2 B)U = 
U^T
       \begin{bmatrix}
      \frac{\alpha+\beta}{2} u_n                      &  \frac{1-\alpha}{2} n_1 \sqrt{U_1} &   \frac{1-\beta}{2} n_2 \sqrt{U_1}     \\
     \frac{1-\alpha}{2} n_1 \sqrt{U_1} & \frac{1}{2} u_n                             &    0                                                   \\
     \frac{1-\beta}{2} n_2 \sqrt{U_1}   & 0                                                  &\frac{1}{2} u_n
       \end{bmatrix}
U=
U^T
       \begin{bmatrix}
      u_n                      & 0                       & 0    \\
                                  & \frac{1}{2} u_n  &    0                                                   \\
       0                         & 0                       &\frac{1}{2} u_n
       \end{bmatrix}
U
\end{equation}
and the (somewhat mysterious) dependency on the free parameters $\alpha$ and $\beta$ vanishes for algebraic reasons. The relation (\ref{boundarmatrix_2D}) seemingly indicate that we need three boundary conditions at inflow ($u_n<0$), and zero at outflow ($u_n>0$). However, this is a nonlinear problem and as shown in  \cite{nordstrom2021linear}, it can be rewritten by changing variables and observing that $u_n=(n_1 U_2+n_2 U_3)/\sqrt{U_1}$. Reformulating  (\ref{boundarmatrix_2D}) in new variables give
\begin{equation}\label{boundarmatrix_2D_SWE_final}
U^T (n_1A+n_2 B)U = 
U^T
       \begin{bmatrix}
      u_n                      & 0                       & 0    \\
                                  & \frac{1}{2} u_n  &    0                                                   \\
       0                         & 0                       &\frac{1}{2} u_n
       \end{bmatrix}
U=
W^T
       \begin{bmatrix}
      -\frac{1}{2 U_n \sqrt{U_1}}                     & 0                       & 0    \\
                                  & \frac{1}{2 U_n \sqrt{U_1}}   &    0                                                   \\
       0                         & 0                       &\frac{1}{2 U_n \sqrt{U_1}} 
       \end{bmatrix}
W,
\end{equation}
where $W^T=(W_1, W_2, W_3)=(U_1^2, U_1^2+U_n^2, U_n U_{\tau})$. The variables $(U_1,U_n,U_{\tau})=(\phi, \sqrt{\phi} u_n,  \sqrt{\phi} u_{\tau})$ are directed in the  normal ($U_n$) and tangential  ($U_{\tau}$) direction respectively. Since we search for a minimal number of boundary conditions, we consider the formulation (\ref{boundarmatrix_2D}) for outflow, where no boundary conditions are required. The relation (\ref{boundarmatrix_2D_SWE_final}) for inflow indicate
that only two boundary conditions are needed at inflow when $U_n<0$. To be specific, at inflow where $U_n < 0$ we find 
\begin{equation}
W^-=
\begin{bmatrix}
U_1^2+U_n^2\\
U_n U_{\tau} 
\end{bmatrix},
\quad
\Lambda^- =
\begin{bmatrix}
 \frac{1}{2 U_n \sqrt{U_1}}   & 0 \\
0 &  \frac{1}{2 U_n \sqrt{U_1}}  
\end{bmatrix},
\quad
W^+= U_1^2,
\quad
\Lambda^+ = -  \frac{1}{2 U_n \sqrt{U_1}}  
\label{SWE_inflow}.
\end{equation}
The definitions in (\ref{SWE_inflow}) can be used to check any inflow conditions for boundedness using Lemma \ref{lemma:GenBC}.

\begin{example}\label{SWE_Example} 
Consider the general form of nonlinear boundary condition in (\ref{Gen_BC_form}).

With Dirichlet like inflow conditions on $U_n, U_{\tau}$ we must remove the influence of  $U_1$. Using (\ref{Gen_BC_form}) we find
\begin{equation}
S^{-1} \left(
\begin{bmatrix}
 \frac{1}{ \sqrt{2 |U_n| \sqrt{U_1}}}   & 0 \\
0 &  \frac{1}{ \sqrt{2 |U_n| \sqrt{U_1}}}  
\end{bmatrix}
\begin{bmatrix}
U_1^2+U_n^2\\
U_n U_{\tau} 
\end{bmatrix}
-
\begin{bmatrix}
R_1 \\
R_2
\end{bmatrix} 
 \frac{1}{ \sqrt{2 |U_n| \sqrt{U_1}}}  
U_1^2 \right) = G
\label{SWE_inflow_ex}.
\end{equation}
In a similar way as for the Dirichlet inflow condition for the IEE we find that this require $R=(1,0)^T $ which in turn lead to $ I- R^T  R=0$. Hence condition  (\ref{R_condition}) is satisfied, but not strictly, which makes the choice of $S$ in (\ref{SWE_inflow_ex})  (and (\ref{S_condition})) irrelevant (similar to the inflow case in Example \ref{IEE_EX}). This leads to boundedness, but not strong boundedness as defined in Proposition \ref{lemma:Matrixrelation}. A weak 
implementation require $\Sigma= \sqrt{|\Lambda^- |} $ in (\ref{SWE_inflow}).

By instead specifying the characteristic variable $W^-$ directly (similar to the outflow case in Example \ref{IEE_EX}) we have $R=(0,0)^T$ and (\ref{R_condition}) holds strictly. This lead to a strongly bounded solution for  $S$ such that $I-S^T S \geq 0$ as required in (\ref{S_condition}).
A weak  implementation require $\Sigma=\sqrt{|\Lambda^- |}$ as given in (\ref{SWE_inflow}), see (\ref{Sigma_condition}).


\end{example} 

\subsection{The 2D compressible Euler equations}\label{CEEex_2D}

The 2D CEEs on the  skew-symmetric form presented in Proposition \ref{lemma:Matrixrelation} were derived in  \cite{Nordstrom2022_Skew_Euler}. They are
\begin{equation}\label{eq:stab_eq_eul}
P \Phi_t + (A  \Phi)_x + A^T  \Phi_x + (B  \Phi)_y + B^T  \Phi_y=0,
\end{equation}
where  $\Phi=(\sqrt{\rho}, \sqrt{\rho} u, \sqrt{\rho} v, \sqrt{p})^T$, $P=diag(1, (\gamma -1)/2, (\gamma -1)/2, 1)$ and
\begin{equation}\label{eq:new_matrix_finalAs}
A = \frac{1}{2}\begin{bmatrix}
      u                                    &  0  &  0 & 0 \\
      0                                    &   \frac{(\gamma -1)}{2} u  &  0 & 0                \\
      0  &   0                                                      &   \frac{(\gamma -1)}{2} u &  0             \\
      0  &  2 (\gamma -1)  \frac{\phi_4}{\phi_1} & 0  &  (2-\gamma)u           
                   \end{bmatrix}, \quad
B = \frac{1}{2}\begin{bmatrix}
      v                                    &  0  &  0 & 0 \\
      0                                    &  \frac{(\gamma -1)}{2} v  &  0 & 0                \\
      0  &   0                                                      &   \frac{(\gamma -1)}{2} v &  0             \\
      0  &  0 & 2 (\gamma -1)  \frac{\phi_4}{\phi_1}   &  (2-\gamma)v          
                   \end{bmatrix}.
 \end{equation}

By rotating the Cartesian velocities to normal and tangential velocities at the boundary, we obtain
\begin{equation}\label{boundarmatrix_contraction_rotated}
\Phi^T (n_1 A + n_2 B) \Phi = 
\Phi^T_r  \begin{bmatrix}
      u_n  &  0 &  0 & 0                       \\
       0                    &  \frac{(\gamma -1)}{2}   u_n         &  0                                          &      (\gamma -1)  \frac{\phi_4}{\phi_1}                                                 \\
       0                    &   0                       &  \frac{(\gamma -1)}{2}   u_n                                                                 &  0                                                      \\
       0                    &  (\gamma -1)  \frac{\phi_4}{\phi_1}& 0 & (2-\gamma)u_n            
                   \end{bmatrix} \Phi_r,
\end{equation}
where  $\Phi_r=(\phi_1, \phi_2, \phi_3,  \phi_4)^T=(\sqrt{\rho}, \sqrt{\rho} u_n, \sqrt{\rho} u_{\tau}, \sqrt{p})^T$. 

The boundary term (\ref{boundarmatrix_contraction_rotated}) can be transformed to diagonal form with the boundary term $W^T \Lambda W$ where
\begin{equation}
W=
\begin{bmatrix}
\phi_1\\
\phi_2+2 \phi_4^2/\phi_2 \\
\phi_3 \\
\phi_4 
\end{bmatrix}
\quad 
\Lambda=
\begin{bmatrix}
      u_n  &  0 &  0 & 0                       \\
       0                    &  \frac{(\gamma -1)}{2}   u_n         &  0                                          &   0                                                \\
       0                    &   0                       &  \frac{(\gamma -1)}{2}   u_n                                                                 &  0                                                      \\
       0                    &  0& 0 & (2-\gamma)u_n     \Psi(M_n)        
                   \end{bmatrix}.
\label{CEE_def}
\end{equation}
The last diagonal entry contain the factor $\Psi(M_n)$ which is a function of the normal Mach number $M_n = u_n/c$. Explicitly we have
\begin{equation}\label{psi_def}
\Psi(M_n)= 1-\frac{2(\gamma -1)}{\gamma (2-\gamma)} \frac{1}{M_n^2}, 
\end{equation}
which switches sign at  $M_n^2=2(\gamma -1)/(\gamma (2-\gamma))$. 
As shown in  \cite{Nordstrom2022_Skew_Euler}, this yields  $|M_n|=1$ for $\gamma=\sqrt{2}$, while for  $\gamma=1.4$ we get $|M_n| \approx 0.95$.

Due to the sign shift in $\Psi$ at $M_n^2=\gamma (2-\gamma)/(2(\gamma -1))$ we get different cases for inflow when $u_n < 0$.  We find that for inflow when $u_n <0, \Psi < 0$, the relation (\ref{CEE_def}) leads to
\begin{equation}
W^-=
\begin{bmatrix}
\phi_1\\
\phi_2+2 \phi_4^2/\phi_2 \\
\phi_3
\end{bmatrix},
\quad 
\Lambda^-=
\begin{bmatrix}
      u_n  &  0 &  0                 \\
       0                    &  \frac{(\gamma -1)}{2}   u_n         &  0                                                \\
       0                    &   0                       &  \frac{(\gamma -1)}{2}   u_n                                                                                     \\
\end{bmatrix},
\quad 
W^+= \phi_4,
\quad
\Lambda^+ = (2-\gamma)u_n     \Psi(M_n).
\label{CEE_inflow}
\end{equation}
For inflow $u_n <0, \Psi >0$ we get $W^-=W$ and $\Lambda^- = \Lambda$ from relation (\ref{CEE_def}), i.e. all eigenvalues are negative. 

In the outflow case, the shift in speed can be ignored since an alternate form of (\ref{boundarmatrix_contraction_rotated}) different from (\ref{CEE_def}) exist. By contracting (\ref{boundarmatrix_contraction_rotated}) we find that 
\begin{equation}\label{contracted_CEE_BT}
\Phi^T (n_x  A + n_y B) \Phi = u_n(\phi_1^2+\frac{(\gamma -1)}{2} (\phi_2^2+\phi_3^2)+\gamma  \phi_4^2)=
\Phi^T_r  \begin{bmatrix}
      u_n  &  0 &  0 & 0                       \\
       0                    &  \frac{(\gamma -1)}{2}   u_n         &  0                                          &    0                                               \\
       0                    &   0                       &  \frac{(\gamma -1)}{2}   u_n                                                                 &  0                                                     \\
       0                    & 0 & 0 & \gamma u_n            
                   \end{bmatrix} \Phi_r,
\end{equation}
which proves that no boundary conditions are necessary in the outflow case. 
\begin{example} Consider the general form of boundary condition in (\ref{Gen_BC_form}).

With Dirichlet inflow conditions on $\phi_1,\phi_2,\phi_3$ for $\Psi(M_n) < 0$ we find using (\ref{CEE_inflow})
\begin{equation}
S^{-1} \left(
\begin{bmatrix}
      \sqrt{|u_n|}  &  0 &  0                 \\
       0                    &  \sqrt{\frac{(\gamma -1)}{2} |u_n|}         &  0                                                \\
       0                    &   0                       &    \sqrt{\frac{(\gamma -1)}{2} |u_n|}                                                 \\
\end{bmatrix}
\begin{bmatrix}
\phi_1\\
\phi_2+2 \phi_4^2/\phi_2 \\
\phi_3
\end{bmatrix}
-
 \sqrt{(2-\gamma)u_n     \Psi(M_n)}
\textcolor{blue}{\begin{bmatrix}
R_1  \\
R_2  \\
R_3
\end{bmatrix}} \phi_4\right) = G
\label{IEE_outflow_ex}
\end{equation}
which lead to \textcolor{blue}{$(R_1,R_2,R_3)^T=(0,R_2,0)^T$} where \textcolor{blue}{$R_2=\sqrt{2(\gamma -1)/((2-\gamma) |\Psi|)}(\phi_4/\phi_2)$}. Hence condition  (\ref{R_condition}) is violated since  $ I- R^T  R= 1- R_2^2=1+ (\Psi-1)/ |\Psi|)=-1/ |\Psi|)<0$ and no bound can be found. 

By specifying the characteristic variable  $W^-$ directly (as in the outflow case in Example \ref{IEE_EX} and the inflow case in Example \ref{SWE_Example}) we have $R=(0,0,0)^T$ and (\ref{R_condition}) holds strictly. This lead to a strongly bounded solution if  $S$ is such that $I-S^T S \geq 0$ as required in (\ref{S_condition}).
 A weak  implementation require $\Sigma=\sqrt{|\Lambda^- |}$ as given in (\ref{CEE_inflow}) and (\ref{Sigma_condition}).


\end{example} 

\subsection{The 2D incompressible Navier-Stokes equations}\label{NSex}
Here we extend the analysis of the IEEs (\ref{NS_splitting}) to the INSEs by adding on the viscous fluxes to obtain 
\begin{equation}
P U_t 
+ \frac{1}{2}\left[ (A U)_x + A U_x 
+  (B U)_y + B U_y \right]
=\epsilon(( P U_x)_x+(P U_y)_y)
\label{INS_splitting}
\end{equation} 
where $U=(u,v,p)^T$,  $\epsilon$ is the non-dimensional viscosity and the matrices $P,A,B$ are given in (\ref{ABI}). The righthand side of (\ref{INS_splitting}) identifies the viscous fluxes. 
The energy method applied to (\ref{INS_splitting}) yields
\begin{equation}\label{NS_BT}
\frac{1}{2} \frac{d}{dt}\|U\|^2_P + \epsilon (\|U_x\|^2_P+\|U_y\|^2_P)  +\frac{1}{2} \oint\limits_{\partial\Omega}U^T  (n_1 A+n_2 B) U - 2\epsilon U^T(n_1 PU_x+n_2PU_y) \\\ ds= 0.
\end{equation}
By using the notations of (\ref{newflux_form_1}) in Section (\ref{nonlinear_BC_visc}), the argument of the surface integral  in (\ref{NS_BT}) becomes
\begin{equation}\label{NS_BT_rotated}
U^T_n \tilde A U_n-  \epsilon (U^T_n  \tilde  F +  \tilde  F^T  U_n) \quad \text{where}  \quad  \tilde  F = ( \tilde  F_n,  \tilde  F_{\tau}, 0)^T
\end{equation}
include the normal $(\tilde  F_n)$ and tangential $(\tilde  F_{\tau})$ shear stresses.

We can now formulate (\ref{NS_BT_rotated}) in the matrix-vector form presented in Section \ref{nonlinear_BC_visc}, which leads to
\begin{equation}\label{matrix-vector_repeated}
\tilde W^T \tilde  \Lambda \tilde W
=
 \begin{bmatrix}
W  - \epsilon  \Lambda^{-1}  T^T  \tilde F\\
              \,\ \quad - \epsilon   \Lambda^{-1} T^T   \tilde F
\end{bmatrix}^T
\begin{bmatrix}
\Lambda & 0 \\
          0 & -   \Lambda
\end{bmatrix}
\begin{bmatrix}
W - \epsilon  \Lambda^{-1}  T^T \tilde F \\
             \,\ \quad  - \epsilon \Lambda^{-1} T^T \tilde F
\end{bmatrix}.
\end{equation}
To determine the minimal number of boundary conditions needed for a bound, we consider the terms related to the lower right block $ (\Lambda^{-1} T^T \tilde F)^T \Lambda (\Lambda^{-1} T^T \tilde F)$ in detail.  That lower block \textcolor{green}{(ignoring the factor $\epsilon$)}  is
\begin{equation}\label{matrix-vector_lower}
\textcolor{blue}{-} (T^T \tilde F)^T \Lambda^{-1} (T^T \tilde F)=
 \begin{bmatrix}
\tilde F_n\\
\tilde F_{\tau} \\
- \tilde F_n/u_n
\end{bmatrix}^T
\begin{bmatrix}
-1/u_n& 0        & 0     \\
 0       & -1/u_n & 0    \\
 0       &0        & u_n
\end{bmatrix}
\begin{bmatrix}
\tilde F_n\\
\tilde F_{\tau} \\
- \tilde F_n/u_n
\end{bmatrix}=-
\tilde F_{\tau}^2/u_n
\end{equation}
which reduces the $6 \times 6$  diagonal matrix in relation (\ref{matrix-vector_repeated}) to a $4 \times 4$ diagonal matrix. We find that
\begin{equation}\label{matrix-vector_repeated_reduced}
\tilde W^T \tilde  \Lambda \tilde W
=
 \begin{bmatrix}
W  - \epsilon  \Lambda^{-1}  T^T  \tilde F\\
              \,\ \quad - \epsilon  \tilde F_{\tau}
\end{bmatrix}^T
\begin{bmatrix}
 \Lambda & 0 \\
0 & -1/u_n 
\end{bmatrix}
\begin{bmatrix}
W - \epsilon  \Lambda^{-1}  T^T \tilde F \\
             \,\ \ - \epsilon \tilde F_{\tau}
\end{bmatrix}.
\end{equation}
In full detail (with a slightly different scaling from the one used in Section \ref{Eulerex}), the result (\ref{matrix-vector_repeated_reduced}) becomes
\begin{equation}\label{matrix-vector_repeated_reduced_full}
\tilde W^T \tilde  \Lambda \tilde W
=
 \begin{bmatrix}
u_n^2 + p- \epsilon  \tilde F_n \\
 u_n u_{\tau}   \   \  \ - \epsilon  \tilde F_{\tau} \\
  \   \  \   \   \  \  \  p  \ - \epsilon  \tilde F_n \\
   \   \  \   \   \  \   \   \  \   \  - \ \epsilon  \tilde F_{\tau}
\end{bmatrix}^T
\begin{bmatrix}
1/u_n & 0 & 0 & 0 \\
0     & 1/u_n & 0 & 0  \\
0      &0     & -1/u_n & 0  \\
0      &0     & 0 & -1/u_n 
\end{bmatrix}
 \begin{bmatrix}
u_n^2 + p- \epsilon  \tilde F_n \\
 u_n u_{\tau}   \   \  \ - \epsilon  \tilde F_{\tau} \\
  \   \  \   \   \  \  \  p  \ - \epsilon  \tilde F_n \\
   \   \  \   \   \  \   \   \  \   \  - \ \epsilon  \tilde F_{\tau}
\end{bmatrix}
\end{equation}
which is exactly what was obtained in \cite{nordstrom2019,nordstrom2020spatial} using a slightly different procedure. 

An inspection of (\ref{matrix-vector_repeated_reduced_full}) reveals that 2 boundary conditions are required for a bound, both at inflow and outflow. 
At inflow where $u_n < 0$ 
\begin{equation}\label{INE_EX_inflow}
W^-=
\begin{bmatrix}
u_n^2 + p- \epsilon  \tilde F_n\\
u_n u_{\tau}   \   \  - \epsilon  \tilde F_{\tau}
\end{bmatrix},
\quad
\Lambda^- =
\begin{bmatrix}
1/u_n & 0 \\
0 & 1/u_n 
\end{bmatrix},
\quad
W^+=
\begin{bmatrix}
p   - \epsilon  \tilde F_n  \\
    \  \   \     - \epsilon  \tilde F_{\tau}
\end{bmatrix},
\quad
\Lambda^+ =
\begin{bmatrix}
-1/u_n & 0 \\
0 & -1/u_n 
\end{bmatrix},
\end{equation}
while at outflow with $u_n > 0$ we get the reversed situation with
\begin{equation}\label{INE_EX_outflow}
W^+=
\begin{bmatrix}
u_n^2 + p- \epsilon  \tilde F_n\\
u_n u_{\tau}   \   \  - \epsilon  \tilde F_{\tau}
\end{bmatrix},
\quad
\Lambda^+ =
\begin{bmatrix}
1/u_n & 0 \\
0 & 1/u_n 
\end{bmatrix},
\quad
W^-=
\begin{bmatrix}
p   - \epsilon  \tilde F_n  \\
    \  \   \     - \epsilon  \tilde F_{\tau}
\end{bmatrix},
\quad
\Lambda^- =
\begin{bmatrix}
-1/u_n & 0 \\
0 & -1/u_n 
\end{bmatrix}.
\end{equation}
By using the definitions (\ref{INE_EX_inflow}) and (\ref{INE_EX_outflow}), it is straightforward to check for boundedness using Lemma \ref{lemma:GenBC}.


\begin{example}\label{IEE_EX} Consider the general form of nonlinear boundary condition in (\ref{Gen_BC_form}).

Dirichlet like inflow conditions on the velocities using the boundary condition (\ref{Gen_BC_form}) and  (\ref{INE_EX_inflow}) become
\begin{equation}
S^{-1} \left(
\begin{bmatrix}
\sqrt{|u_n|} & 0 \\
0 &/\sqrt{|u_n|} 
\end{bmatrix}^{-1} 
\begin{bmatrix}
u_n^2 + p- \epsilon  \tilde F_n\\
u_n u_{\tau}   \   \  - \epsilon  \tilde F_{\tau}
\end{bmatrix}
-
\begin{bmatrix}
R_{11} & R_{12}  \\
R_{21} & R_{22}
\end{bmatrix}
\begin{bmatrix}
\sqrt{|u_n|} & 0 \\
0 &\sqrt{|u_n|} 
\end{bmatrix}^{-1} 
\begin{bmatrix}
p   - \epsilon  \tilde F_n  \\
    \  \       - \epsilon  \tilde F_{\tau}
\end{bmatrix}
  \right) = G.
\label{INE_inflow_ex}
\end{equation}
The matrix $R$ with elements $R_{11}=R_{22}=1$ and $R_{12}=R_{21}=0$ removes the influence of the pressure $p$ and shear stresses. This lead to $I - R^T R=0$ and hence condition  (\ref{R_condition}) is satisfied, but not strictly. This makes the choice of $S$ in (\ref{IEE_inflow_ex}) (and (\ref{S_condition})) irrelevant since only the homogeneous case is possible. This leads to boundedness, but not strong boundedness as defined in Proposition \ref{lemma:Matrixrelation}. A weak 
implementation require $\Sigma= \sqrt{|\Lambda^- |} =diag( 1/\sqrt{|u_n|}, 1/\sqrt{|u_n|})$ as specified in (\ref{Sigma_condition}).

The outflow conditions on the characteristic variables containing the pressure and shear stresses, the boundary condition (\ref{Gen_BC_form}) and  (\ref{INE_EX_outflow}) leads to the reversed formulation
\begin{equation}
S^{-1} \left(
\begin{bmatrix}
\sqrt{|u_n|} & 0 \\
0 &/\sqrt{|u_n|} 
\end{bmatrix}^{-1} 
\begin{bmatrix}
p   - \epsilon  \tilde F_n  \\
    \  \        - \epsilon  \tilde F_{\tau}
\end{bmatrix}
-
\begin{bmatrix}
R_{11} & R_{12}  \\
R_{21} & R_{22}
\end{bmatrix}
\begin{bmatrix}
\sqrt{|u_n|} & 0 \\
0 &\sqrt{|u_n|} 
\end{bmatrix}^{-1} 
\begin{bmatrix}
u_n^2 + p- \epsilon  \tilde F_n\\
u_n u_{\tau}   \   \  - \epsilon  \tilde F_{\tau}
\end{bmatrix}
  \right) = G.
\label{INE_outflow_ex}
\end{equation}
Here we can choose the matrix $R$ to be the null matrix and and hence condition (\ref{R_condition})  holds strictly.  This leads to a strongly energy bounded solution for  $S$ such that $I-S^T S \geq 0$ as required in (\ref{S_condition}) for a vanishing  $R$. A weak  implementation require  $\Sigma= \sqrt{|\Lambda^- |} $  as specified in (\ref{Sigma_condition}).
\end{example} 

\section{Some open questions for nonlinear boundary conditions}\label{sec:open_q}
We end this paper by discussing some open questions arising from the nonlinear system analysis above.

\subsection{The number of boundary conditions in nonlinear IBVPs required for boundedness}\label{sec:open_q_number}
The boundary conditions for the SWEs and CEEs are similar in the sense that at least two {\it different} formulations of the boundary terms can be found. The minimal number of required conditions differ both in the inflow and outflow cases. One common feature is that no outflow conditions seem to be necessary. Another similar feature is that the number of outflow conditions is independent of the speed of sound for the CEEs and the celerity in the SWE case.  Both these effects differ from what one finds in a linear analysis. 
\begin{remark}
By substituting the IEE variables $W=(u_n +p/u_n, u_{\tau},p/u_n)^T$ in (\ref{1Dprimalstab_trans_final_icomp}) with $W=(u_n, u_{\tau}, \sqrt{p})^T$ (similar to the ones used in the CEE and SWE cases) one obtains a similar situation also for the IEEs. The eigenvalues for the IEEs transform from $\Lambda=diag(u_n,u_n,-u_n)$ to $\Lambda=diag(u_n,u_n, 2u_n)$ which leads to different number of boundary conditions.
\end{remark}

\subsection{The effect of nonlinear boundary conditions on uniqueness and existence}\label{sec:open_q_uniqueness}
As stated in Remark \ref{explain_min_nr},  a minimal number of dissipative boundary conditions in the linear case  leads to uniqueness since it determines the normal modes of the solution \cite{kreiss1970,Strikwerda1977797}. The (same) minimal number of boundary conditions can also be obtained using the energy method and specifying the number of boundary conditions required for a bound, see \cite{nordstrom2020}.  If uniqueness and boundedness for a minimal number of boundary conditions are given, existence can be shown (e.g. using Laplace transforms or difference approximations \cite{gustafsson1995time,kreiss1989initial}). For linear IBVPs,  the number of boundary conditions is independent of the solution and only depend on known external data. For nonlinear IBVPs, that is no longer the case, and the number may change as the solution develops in time. In addition, as we have seen above, it also varies depending on the particular formulation chosen. This is confusing and raises a number of questions that we will {\it speculate} on below.


Let us consider the SWEs as an example. The two forms of the boundary terms given in (\ref{boundarmatrix_2D_SWE_final}) were

\begin{equation}\label{boundarmatrix_examples}
U^T
       \begin{bmatrix}
      u_n                      & 0                       & 0    \\
                                  & \frac{1}{2} u_n  &    0                                                   \\
       0                         & 0                       &\frac{1}{2} u_n
       \end{bmatrix}
U=
W^T
       \begin{bmatrix}
      -\frac{1}{2 U_n \sqrt{U_1}}                     & 0                       & 0    \\
                                  & \frac{1}{2 U_n \sqrt{U_1}}   &    0                                                   \\
       0                         & 0                       &\frac{1}{2 U_n \sqrt{U_1}} 
       \end{bmatrix}
W,
\end{equation}
where $W^T=(W_1, W_2, W_3)=(U_1^2, U_1^2+U_n^2, U_n U_{\tau})$ and $(U_1,U_n,U_{\tau})=(\phi, \sqrt{\phi} u_n,  \sqrt{\phi} u_{\tau})$. 
Based on the two formulations in (\ref{boundarmatrix_examples}), one may base the boundary procedure on one of the following four scenarios.
\begin{enumerate}

\item The left formulation with variable $U$ at both inflow and outflow boundaries.

\item The right formulation with variable $W$ at both inflow and outflow boundaries.

\item The left formulation with variable $U$ at inflow and the right formulation with $W$ at outflow boundaries.

\item The right formulation with variable $W$ at inflow and the left formulation with $U$ at outflow boundaries.

\end{enumerate}
Scenario 1 would in a one-dimensional setting lead to three boundary conditions all applied on the inflow boundary. Scenario 2 would also give three boundary conditions, but now two would be applied on the inflow boundary and one on the outflow boundary. Scenario 3 would lead to four boundary conditions, three on the inflow and one on the outflow boundary. Scenario 4 would only give two boundary conditions, both applied on the inflow boundary. 

If the above scenarios were interpreted in the linear sense, both Scenario 1 and 2 would determine the solution uniquely. (One of them would be a better choice than the other depending on the growth or decay of the solution away from the boundary \cite{kreiss1970,Strikwerda1977797}.) In scenario 3, the solution would be overspecified, leading to loss of existence. In scenario 4, the solution would be underspecified, leading to loss of uniqueness. In summary: Scenario 1 and 2 may lead to acceptable solutions, Scenario 3 give no solution at all, while scenario 4 yield a bounded solution with limited (or no) accuracy.

However, since these results are nonlinear, the above summary is merely {\it speculative}. We do not know exactly how to interpret them, since the present nonlinear theory is incomplete. We only know that boundedness is required. It also seems likely though that scenario 1 and 2 should be preferred over scenario 3 and 4. The speculations in this section are of course equally valid (or not valid) for the CEEs and IEEs.
\begin{remark}
Whether a reformulation that leads to different minimal numbers of boundary conditions can be done also in the second derivative cases, e.g. in the INSEs,  is not clear. We think not, since the added viscous terms are  linear and do not contribute to the diagonal entries.
\end{remark}

\section{Summary}\label{sec:conclusion}
In this paper we have complemented the stability analysis of nonlinear initial boundary value problems on skew-symmetric first derivative form initiated in  \cite{nordstrom2022linear-nonlinear,Nordstrom2022_Skew_Euler}, by adding the analysis of nonlinear boundary conditions.
 \textcolor{green}{The new nonlinear boundary procedure for non-zero data generalise the well known characteristic boundary procedure for linear problems to the nonlinear setting.} 
We also extend the analysis to parabolic and incompletely parabolic problems and derive explicit boundary conditions for these cases. The general boundary procedure for nonlinear initial boundary value problems is exemplified using: the shallow water equations, the incompressible Euler equations and Navier-Stokes equations and the compressible Euler equations. Once the bound of the continuous problem is obtained, we show how to implement the boundary conditions in a provable stable way using summation-by-parts formulations and weak boundary procedures. 


\section*{Acknowledgment}
Jan Nordstr\"{o}m was supported by Vetenskapsr{\aa}det, Sweden [award no.~2021-05484 VR] and University of Johannesburg.

\bibliographystyle{elsarticle-num}
\bibliography{References_Jan,References_andrew,References_Fredrik}

\end{document}